\documentclass[12pt,reqno]{amsart}
\usepackage[margin=2.2cm]{geometry}

\usepackage{amssymb,amsfonts,bbm}
\usepackage[all,arc]{xy}
\usepackage{enumerate}
\usepackage{mathrsfs}
\usepackage{color}   
\usepackage{hyperref}
\hypersetup{
    colorlinks=true, 
    linktoc=all,     
    linkcolor=blue,  
    citecolor=red,
}
\usepackage{amsthm}
\usepackage{amsmath}
\usepackage{graphicx}
\usepackage{pgf,tikz}
\usepackage{mathrsfs}
\usetikzlibrary{positioning}

\newcommand{\bb}{\mathbb B}
\newcommand{\bc}{\mathbb C}

\newcommand{\bk}{\mathbb K}
\newcommand{\bm}{\mathbb M}
\newcommand{\bn}{\mathbb N}
\newcommand{\bp}{\mathbb P}
\newcommand{\bq}{\mathbb Q}

\newcommand{\bt}{\mathbb T}
\newcommand{\bu}{\mathbb U}
\newcommand{\bz}{\mathbb Z}

\newcommand{\bw}{\mathbf{w}}
\newcommand{\sr}{\boldsymbol{\sigma}}

\newcommand{\obg}{\overline{\bg}}

\newcommand{\oA}{\overline{A}}
\newcommand{\oB}{\overline{B}}
\newcommand{\oG}{\overline{G}}
\newcommand{\oT}{\overline{T}}
\newcommand{\oH}{\overline{H}}
\newcommand{\oM}{\overline{M}}
\newcommand{\oP}{\overline{P}}

\newcommand{\og}{\overline{g}}
\newcommand{\oh}{\overline{h}}
\newcommand{\ot}{\overline{t}}

\newcommand{\oz}{\overline{z}}

\newcommand{\ochi}{\overline{\chi}}

\newcommand{\bfc}{\mathbf{c}}

\newcommand{\bfg}{\mathbf g}
\newcommand{\bfs}{\mathbf{s}}

\newcommand{\Ftn}{\mathrm{Ftn}}
\newcommand{\id}{\mathrm{id}}
\newcommand{\Wh}{\mathrm{Wh}}

\newcommand{\bg}{\mathbb{G}}
\newcommand{\bgl}{\mathbb{GL}}
\newcommand{\bsl}{\mathbb{SL}}

\newcommand{\bfgg}{\mathbf g}
\newcommand{\cw}{\mathcal{W}}

\newcommand{\la}{\langle}
\newcommand{\ra}{\rangle}

\newcommand{\bs}{\backslash}

\newcommand{\al}{\alpha}
\newcommand{\lam}{\lambda}

\DeclareMathOperator{\GL}{GL}
\DeclareMathOperator{\Hom}{Hom}
\DeclareMathOperator{\Ind}{Ind}

\DeclareMathOperator{\SL}{SL}

\newtheorem{Thm}{Theorem}[section]
\newtheorem{Prop}[Thm]{Proposition}
\newtheorem{Lem}[Thm]{Lemma}
\newtheorem{Cor}[Thm]{Corollary}

\theoremstyle{definition}
\newtheorem{Def}[Thm]{Definition}

\theoremstyle{remark}
\newtheorem{Rem}[Thm]{Remark}
\newtheorem{Ex}[Thm]{Example}

\theoremstyle{definition}

\title{Unramified Whittaker functions for certain Brylinski-Deligne covering groups}
\author{Yuanqing Cai}

\date\today
\address{Department of Mathematics, Weizmann Institute of Science, Rehovot, 7610001, Israel}
\email{yuanqing.cai@weizmann.ac.il}
\thanks{This research was supported by the ERC, StG grant number 637912.}

\begin{document}

\begin{abstract}
For a Brylinski-Deligne covering group of a general linear group, we calculate some values of unramified Whittaker functions for certain representations that are analogous to the theta representations.
\end{abstract}

\maketitle
\tableofcontents


\section{Introduction}

The unramified Whittaker functions and their analogues play an important
role in modern number theory, arising naturally as terms in the Fourier coefficients of automorphic forms. It is generally desirable to calculate explicit values for these functions, as the information proves useful in many aspects of study related to the automorphic form (for example, in
the construction of associated $L$-functions). When an automorphic representation possesses a Whittaker model or another suitable unique model, the method
described in \cite{CS80} may be used to compute an explicit formula (the Casselman-Shalika formula) for the values of the unramified Whittaker function (or the analogous function). 

In this paper, we consider representations of Brylinski-Deligne covering groups.
For these groups, the uniqueness of Whittaker models fails in general. This causes obstructions to some advancement of the theory. Nevertheless, in the past decades, it is discovered that Fourier coefficients of Eisenstein series on covering groups are closely tied to the Weyl group multiple Dirichlet series. This leads to several generalizations of the Casselman-Shalika formula to the covering group setup.  One is to interpret the value of an unramified Whittaker function as a weighted sum over a crystal graph. In this vein, this beautiful idea is realized in \cite{BBF11,McNamara11,FZ15} for root systems of type $A$ and $C$. The other description is to express the value as the average of a Weyl group action. This approach is closer to the one of Casselman-Shalika and is successful for all types of root systems (see \cite{CO13,McNamara16,CG10}). In the linear case, the equivalence of these two descriptions is a formula of Tokuyama.

However, the formulas mentioned above are not explicit to work with. To seek applications towards the theory of automorphic forms on covering groups, we would like to have a formula analogous to the original Casselman-Shalika formula. At the moment, we believe that this is impossible in general. Thus, in this paper, we would like to consider the following weaker question:
\begin{itemize}
\item For representations on covering groups with additional features (for example, theta representations), is it possible to give a simple formula for some values of the unramified Whittaker functions?
\end{itemize}

In this paper, we address this question for Brylinski-Deligne covering groups of general linear groups. We give an answer to this question for a family of representations, that can be viewed as analogues of the theta representations. Such representations were also studied in \cite{Suzuki97,Suzuki98}, and a formula was successfully obtained in some cases. Our results generalize part of Suzuki's results.

Let $\bg=\bgl_r$ over a local non-Archimedean field $F$ and $\oG$ be the degree $n$ Brylinski-Deligne covering group arising from a $\bk_2$-extension $\obg$ of $\bg$. Let $\bm$ be a Levi subgroup of $\bg$. Let $I(\ochi)$ be an unramified principle series representation of $\oG$. Suppose that $\ochi$ is an ``anti-exceptional character in $\bm$'' (Definition \ref{def:exceptional}). Let $\bw_M$ be the long element in the Weyl group $W(\bm)$. Define $\Theta(\oG/\oM,\ochi)$ as the image of the intertwining operator $I(\ochi)\to I({}^{\bw_M}\ochi)$ (Sect. \ref{sec:relative theta}). Let $\cw_0(\og,\ochi)$ be an unramified Whittaker function in a certain Whittaker model of $\Theta(\oG/\oM,\ochi)$. Let $e$ be the identity element in $\oG$.

\begin{Thm}[Theorem \ref{thm:final formula}]
With the above notations and certain assumptions on the rank of $\bm$ and the degree of $\oG$, $\cw_0(e,\ochi)$ is a product of a certain Gauss sum and a polynomial in terms of `Satake parameters' of $I(\ochi)$.
\end{Thm}

When $\bm=\bg$, then $\Theta(\oG/\oM,\ochi)$ is the theta representation studied in \cite{KP84,Gao17}. When $\bm$ has up to two factors, such results are obtained in \cite{Suzuki97,Suzuki98}. Our proof uses ideas in these two papers.

To generalize the results in Suzuki's papers to our setup, another idea is required. That is to utilize the crystal graph description as a key input. This idea was already used in \cite{Kaplan} Theorem 43. Here we extend it to a slightly more general setup.

For small rank symplectic groups, similar formulas were obtained in \cite{Gao}. It will be interesting to see whether the method in this paper can be extended to other groups.


We now give an outline of this paper. Sect. \ref{sec:preliminaries} gives preliminary results on the Brylinski-Deligne covering groups. We introduce the unramified principal series representations and the Casselman-Shalika formula in Sect. \ref{sec:unramified principal}. We then prove an inductive formula for unramified Whittaker functions in Sect. \ref{sec:inductive}. Such results were obtained by Suzuki in type $A$ and here we extend it to all types. We then introduce the representation $\Theta(\oG/\oM,\ochi)$, which we call the relative theta representation (Sect. \ref{sec:relative theta}). In Sect. \ref{sec:general linear}, we specialize our results to the case of general linear groups. We calculate a crucial local matrix coefficient in Sect. \ref{sec:calculation of local matrix}. This is where the ideas of Suzuki are used. In Sect. \ref{sec:main result}, we state our main results and give a proof. We also add simple examples to help the reader understand the ideas. As the area of covering groups is of deep nature, we either give reliable references or reproduce the necessary proofs here. We also try to fill gaps in past literatures as much as possible.

\subsection*{Acknowledgments}
The author would like to thank Solomon Friedberg and Eyal Kaplan for explaining to him that his original approach did not work. The author would also like to thank the Institute for Mathematical Sciences at the National University of Singapore, where part of this work was done during a visit from December 2018 to January 2019.

\section{Preliminaries}\label{sec:preliminaries}

We first recall some structural facts on the Brylinski-Deligne covering groups \cite{BD01, GG18}. In this paper, we concentrate exclusively on unramified Brylinski-Deligne covering groups. We use \cite{Gao17} as our main reference.

\subsection{$\bk_2$-extensions}

Let $F$ be a non-Archimedean field of characteristic $0$, with residual characteristic $p$. Let $O_F$ be the ring of integers. Fix a uniformizer $\varpi$ of $F$. Let $\bg$ be a split connected linear algebraic group over $F$ with maximal split torus $\bt$. Let
\[
\{
X, \ \Phi, \  \Delta; \  Y, \ \Phi^\vee, \ \Delta^\vee
\}
\]
be the based root datum of $\bg$. Here $X$ (resp. $Y$) is the character lattice (resp. cocharacter lattice) for $(\bg,\bt)$. Choose a set $\Delta\subset \Phi$ of simple roots from the set of roots $\Phi$, and $\Delta^\vee$ the corresponding simple coroots from $\Phi^\vee$. Write $Y^{sc}\subset Y$ for the sublattice generated by $\Phi^\vee$. Let $\bb=\bt\bu$ be the Borel subgroup associated with $\Delta$. Denote by $\bu^-\subset \bg$ the unipotent subgroup opposite to $\bu$.

Fix a Chevalley system of pinnings for $(\bg,\bt)$, that is, we fix a set of compatible isomorphisms
\[
\{e_\al:\bg_a\to \bu_\al\}_{\al\in\Phi},
\]
where $\bu_\al\subset \bg$ is the root subgroup associated with $\al$. In particular, for each $\al\in\Phi$, there is a unique morphism $\varphi_\al:\bsl_2\to\bg$ which restricts to $e_{\pm \al}$ on the upper and lower triangular subgroup of unipotent matrixes of $\bsl_2$.

Denote by $W=W(\bg)$ the Weyl group of $(\bg,\bt)$, which we identify with the Weyl group of the coroot system. In particular, $W$ is generated by simple reflections $\{\sr_\al:\al^\vee\in\Delta^\vee\}$ for $Y\otimes \bq$. Let $\ell:W\to \bn$ be the length function. Let $\bw_G$ be the longest element in $W$.

Consider the algebro-geometric $\bk_2$-extension $\obg$ of $\bg$, which is categorically equivalent to the pairs $\{(D,\eta)\}$ (see \cite{GG18} Section 2.6). Here $\eta:Y^{sc}\to F^\times$ is a homomorphism. On the other hand,
\[
D:Y\times Y\to \bz
\]
is a  bisector associated to a Weyl-invariant quadratic form $Q:Y\to \bz$. That is, let $B_Q$ be the Weyl-invariant bilinear form associated to $Q$ such that
\[
B_Q(y_1,y_2)=Q(y_1+y_2)-Q(y_1)-Q(y_2),
\]
then $D$ is a bilinear form on $Y$ satisfying
\[
D(y_1,y_2)+D(y_2,y_1)=B_Q(y_1,y_2).
\]
The bisection $D$ is not necessarily symmetric. Any $\obg$ is, up to isomorphism, incarnated by (i.e. categorically associated to) a pair $(D,\eta)$ for a bisector $D$ and $\eta$.

\subsection{Topological covering}

Let $n\geq 1$ be a natural number. Assume that $F^\times$ contains the full group $\mu_{2n}$ of $2n$-th roots of unity and $p\nmid n$. With this assumption, $(\varpi,\varpi)_n=1$ for the Hilbert symbol $(\cdot,\cdot)_n$. This fact is crucial for several results later.

Let $\obg$ be incarnated by $(D,\eta)$. One naturally obtains degree $n$ topological covering groups $\oG,\oT,\oB$ of rational points $G:=\bg(F), T:=\bt(F), B:=\bb(F)$, such as
\[
\mu_n\hookrightarrow \oG\twoheadrightarrow G.
\]
We may write $\oG^{(n)}$ for $\oG$ to emphasize the degree of covering. For any subset $H\subset G$, we write $\oH\subset \oG$ for the preimage of $H$ with respect to the quotient map $\oG\to G$. The Bruhat-Tits theory gives a maximal compact subgroup $K\subset G$, which depends on the fixed pinnings. We assume that $\oG$ splits over $K$ and fixes such a splitting; the group $\oG$ is called an unramified Brylinski-Deligne covering group in this case. We remark that if the derived group of $\bg$ is simply connected, then $\oG$ splits over $K$ (see \cite{GG18} Theorem 4.2). On the other hand, there is a certain double cover of $\mathrm{PGL}_2$ where the splitting does not exist (see \cite{GG18}, Sect. 4.6).

The data $(D,\eta)$ play the following role for the structural fact on $\oG$:

\begin{itemize}
\item The group $\oG$ splits canonically over any unipotent element of $G$. In particular, we write $\overline{e}_\al(u)\in\oG, \ \al\in \Phi, \ u\in F$ for the canonical lifting of $e_\al(u)\in G$. For any $\al\in\Phi$, there is a natural representative $\sigma_\al:=e_\al(1)e_{-\al}(-1)e_\al(1)\in K$ (and therefore $\overline{\sigma}_\al\in\oG$ by the splitting of $K$) of the Weyl element $\sr_\al\in W$. For a general Weyl group element $\bw$, one can find a lift $w\in\oG$ based on a reduced decomposition of $\bw$. This lift does not depend on the choice of reduced decomposition. We refer to \cite{Gao18} Sect. 6.1 for a detailed discussion on this matter. Moreover, for $h_\al(a):=\al^\vee(a)\in G, \al\in\Phi, a\in F^\times$, there is a natural lifting $\overline{h}_\al(a)\in\oG$ of $h_\al(a)$, which depends only on the pinning and the canonical unipotent splitting (\cite{GG18} Sect. 4.6).
\item There is a section $\bfs$ of $\oT$ over $T$ such that the group law on $\oT$ is given by
\[
\bfs(y_1(a))\cdot \bfs(y_2(b))=(a,b)_n^{D(y_1,y_2)}\cdot \bfs(y_1(a)\cdot y_2(b)).
\]
Moreover, for the natural lifting $\overline{h}_\al(a)$, one has
\[
\overline{h}_\al(a)=(\eta(\al^\vee),a)_n \cdot \bfs(h_\al(a))\in\oT.
\]
\item Let $\sr_\al\in G$ be the natural representative of $\sr_\al\in W$. For any $\overline{y(a)}\in\oT$,
\[
\sr_\al \cdot \overline{y(a)} \cdot \sr_\al^{-1}=\overline{y(a)}\cdot \overline{h}_\al(a^{-\la y,\al\ra}),
\]
where $\la \ , \ \ra$ is the pairing between $Y$ and $X$.
\end{itemize}

We recall the following lemma.

\begin{Lem}[\cite{Gao18b} Lemma 2.1]
For all $y\in Y$,
\[
w\cdot \bfs_y \cdot w^{-1}=\bfs_{\bw(y)}.
\]
\end{Lem}

Define the sublattice
\[
Y_{Q,n}:=
\{
y\in Y: B_Q(y,y')\in n\bz
\}
\]
of $Y$. For every $\al^\vee\in\Phi^\vee$, define
\[
n_\al:=n/\gcd(n,Q(\al^\vee)).
\]
Write $\al_{Q,n}^\vee:=n_\al \al^\vee, \al_{Q,n}:=n_\al^{-1}\al$.
Let $Y_{Q,n}^{sc}\subset Y$ be the sublattice generated by $\{\al_{Q,n}^\vee\}_{\al\in\Phi}$. The complex dual group $\oG^\vee$ for $\oG$ as given in \cite{FL10, McNamara12, Reich12} has root data
\[
(Y_{Q,n}, \ \{ \al_{Q,n}^\vee \}, \ \Hom(Y_{Q,n},\bz), \ \{\al_{Q,n}\}).
\]
In particular, $Y_{Q,n}^{sc}$ is the root lattice for $\oG^\vee$. 


\subsection{Gauss sum}

Consider the Haar measure $\mu$ of $F$ such that $\mu(O_F)=1$. Thus,
\[
\mu(O_F^\times)=1-1/q.
\]
The Gauss sum is given by
\[
G_\psi(a,b)=\int_{O_F^\times} (u,\varpi)_n^a \cdot \psi(\varpi^b u) \ \mu(u), \qquad a,b\in\bz.
\]
It is known that
\[
G_\psi(a,b)=
\begin{cases}
  0, & \mbox{if } b<-1 \\
  1-1/q, & \mbox{if } n\mid a, \ b\geq 0 \\
  0, & \mbox{if } n\nmid a, \ b\geq 0 \\
  -1/q, & \mbox{if } n\mid a, \ b=-1 \\
  G_\psi(a,-1) \text{ with } |G_\psi(a,-1)|=q^{-1/2}, & \mbox{if }n\nmid a, \ b=-1.
\end{cases}
\]
Let $\varepsilon:=(-1,\varpi)_n\in\bc^\times$. One has $\overline{G_\psi(a,b)}=\varepsilon^a\cdot G_\psi(-a,b)$. For any $k\in \bz$, we write
\[
\bfgg_\psi(k):=G_\psi(k,-1).
\]

\subsection{Actions}

Let $\rho=\frac{1}{2}\sum_{\al\in\Phi^+}\al^\vee$. We define an action of $W$ on $Y\otimes \bq$, which we denote by $\bw[y]$ by
\[
\bw[y]:=\bw(y-\rho)+\rho.
\]
If we write $y_\rho:=y-\rho$ for any $y\in Y$, then $\bw[y]-y=\bw(y_\rho)-y_\rho$. From now on, by Weyl orbits in $Y$ or $Y\otimes \bq$ we always refer to the ones with respect to the action $\bw[y]$.
Note that here $0\in Y$ is a vector. The size of this vector is always clear in the context, and we hope that this does not arise any confusion.




We now list some other notations which appear frequently in the text:
\begin{itemize}



\item $\psi$: a fixed additive character of $F\to \bc^\times$ with conductor $O_F$. For any $a\in F^\times$, the twisted character $\psi_a$ is given by
\[
\psi_a:x\mapsto \psi(ax).
\]

\item $\bfs_y:$ for any $y\in Y$, we write $\bfs_y:=\bfs(\varpi^y)\in\oT$.

\item $\lceil x \rceil$: the minimum integer such that $\lceil x \rceil \geq x$ for a real number $x$.
\item $\lfloor x \rfloor$: the maxmial integer such that $\lfloor x \rfloor \leq x$ for a real number $x$.
\item $\ochi_\al$: for an unramified character $\ochi$, we sometimes write $\ochi_\al=\ochi(\oh_\al(\varpi^{n_\al}))$.
\item $y\sim y'$: if $y,y'\in Y$, we write $y\sim y'$ if there exists $\bw \in W$ such that $y'=\bw[y]$.
\item $\id\in W$: the identity element in $W$.
\item $\ochi\sim (\ochi_1,\cdots,\ochi_k)$: see Sect. \ref{sec:local mat coeff for Levi}.
\end{itemize}

\section{Unramified principal series representations}\label{sec:unramified principal}

Fix an embedding $\iota:\mu_n\hookrightarrow \bc^\times$. A representation of $\oG$ is called $\iota$-genuine if $\mu_n$ acts via $\iota$. We consider throughout the paper $\iota$-genuine (or simply genuine) representations of $\oG$.

Let $U$ be the unipotent subgroup of $B=TU$. As $U$ splits canonically in $\oG$, we have $\oB=\oT U$. The covering torus $\oT$ is a Heisenberg group with center $Z(\oT)$. The image of $Z(\oT)$ in $T$ is equal to the image of the isogeny $Y_{Q,n}\otimes F^\times \to T$ induced from $Y_{Q,n}\to Y$.

Let $\ochi\in \Hom_\iota(Z(\oT),\bc^\times)$ be a genuine character of $Z(\oT)$. Write $i(\ochi):=\Ind_{\oA}^{\oT}\ochi'$ for the induced representation on $\oT$, where $\oA$ is any maximal abelian subgroup of $\oT$, and $\ochi'$ is any extension of $\ochi$. By the Stone-von Neumann theorem (see \cite{Weissman09} Theorem 3.1, \cite{McNamara12} Theorem 3), the construction $\ochi\mapsto i(\ochi)$ gives a bijection between isomorphism classes of genuine representations of $Z(\oT)$ and $\oT$. Since we consider an unramified covering group $\oG$ in this paper, we take $\oA$ to be $Z(\oT)\cdot (K\cap T)$ from now on.

The choice of this maximal abelian group here is crucial for our calculation in Sect. \ref{sec:main result}.

\subsection{Definition}

View $i(\ochi)$ as a genuine representation of $\oB$ by inflation from the quotient map $\oB\to \oT$. We now define the unramified principal series representation $I(\chi):=\Ind_{\oB}^{\oG}i(\chi)$. The induction is normalized. 
One knows that $I(\ochi)$ is unramified (i.e. $I(\ochi)^K\neq 0$) if and only if $\ochi$ is unramified (i.e. $\ochi$ is trivial on $Z(\oT)\cap K$). We only consider unramified genuine representations in this paper. One has the natually arising abelian extension
\[
\mu_n\hookrightarrow \overline{Y}_{Q,n}\twoheadrightarrow Y_{Q,n}
\]
such that unramified genuine characters $\ochi$ of $Z(\oT)$ correspond to genuine characters of $\overline{Y}_{Q,n}$. Here $\overline{Y}_{Q,n}=Z(\oT)/Z(\oT)\cap K$. Since $\oA/(T\cap K)\simeq \overline{Y}_{Q,n}$ as well, there is a canonical extension (also denoted by $\ochi$) of an unramified character $\ochi$ of $Z(\oT)$ to $\oA$, by composing $\ochi$ with $\oA\twoheadrightarrow \overline{Y}_{Q,n}$. Therefore, we will identity $i(\ochi)$ as $\Ind_{\oA}^{\oT}\ochi$ with this $\ochi$.

The following result appears in the proof of \cite{McNamara12} Lemma 2.

\begin{Lem}\label{lem:supp of spherical}
An unramified principal series representation $I(\chi)$ has a one-dimensional space of $K$-fixed vectors. There is an isomorphism
\[
i(\chi)^{\oT\cap K}\simeq I(\ochi)^{K}.
\]
Given $f\in i(\chi)^{\oT\cap K}$, the support of $f$ is in $\oA$.
\end{Lem}

For any $\bw\in W$, the intertwining operator $T_{\bw,\ochi}:I(\ochi)\to I({}^{\bw}\ochi)$ is defined by
\[
(T_{\bw,\ochi}f)(\og)=\int_{U_w} f(w^{-1} u \og) \ du
\]
when it is absolutely convergent. Here, $U_w=U\cap wU^-w^{-1}$. Moreover, it can be meromorphically continued for all $\ochi$ (\cite{McNamara12} Sect. 7). For $I(\ochi)$ unramified and $\bw=\sr_\al$ with $\al\in\Delta$, $T_{\sr_\al,\ochi}$ is determined by
\[
T_{\sr_\al,\ochi}(\phi_K)=c(\sr_\al,\ochi)\cdot \phi_K^{\sr_\al}
\]
where
\[
c(\sr_\al,\ochi)=\dfrac{1-q^{-1}\ochi(\oh_\al(\varpi^{n_\al}))}{1-\ochi(\oh_\al(\varpi^{n_\al}))}.
\]
Here $\phi_K\in I(\ochi)$ and $\phi_K^{\sr_\al}\in I({}^{\sr}\ochi)$ are the normalized unramified vectors (\cite{McNamara12,Gao18}).

For a general $\bw \in W$, denote
\[
\Phi(\bw):=\{\al\in \Phi:\al>0 \text{ and }\bw(\al)<0\}.
\]
Then the Gindikin-Karpelevich coefficient $c(\bw,\chi)$ associated with $T_{\bw,\chi}$ is
\[
c(\bw,\chi)=\prod_{\al\in\Phi(\bw)} c(\sr_\al,\chi)
\]
such that $T_{\bw,\chi}(\phi_K)=c(\bw,\ochi)\phi'_K$.

\subsection{Whittaker functional}

Let $\Ftn(i(\ochi))$ be the vector space of functions $\bfc$ on $\oT$ satisfying
\[
\bfc(\ot\cdot \oz)=\bfc(\ot)\cdot \ochi(\oz), \qquad \ot\in\oT \text{ and } \oz\in\oA.
\]
The support of any $\bfc\in\Ftn(i(\ochi))$ is a disjoint union of cosets in $\oT/\oA$. Moreover $\dim(\Ftn(i(\ochi)))=|Y/Y_{Q,n}|$ since $\oT/\oA$ has the same size as $Y/Y_{Q,n}$.

There is a natural isomorphism of vector spaces $\Ftn(i(\ochi))\simeq i(\ochi)^\vee$, where $i(\ochi)^\vee$ is the complex dual space of functionals of $i(\ochi)$. Explicitly, let $\{\gamma_i\}\subset \oT$ be a set of representatives of $\oT/\oA$. Consider $\bfc_{\gamma_i}\in\Ftn(i(\ochi))$ which has support $\gamma_i\cdot\oA$ and $\bfc_{\gamma_i}(\gamma_i)=1$. It gives rise to a linear functional $\lam_{\gamma_i}^{\ochi}\in i(\ochi)^\vee$ such that
\[
\lam_{\gamma_i}^{\ochi}(f_{\gamma_j})=\delta_{ij},
\]
where $f_{\gamma_j}\in i(\ochi)$ is the unique element such that $\mathrm{supp}(f_{\gamma_j})=\oA\cdot \gamma_j^{-1}$ and $f_{\gamma_j}(\gamma_j^{-1})=1$.
That is,
\[
f_{\gamma_j}=i(\ochi)(\gamma_j)\phi_K.
\]
The isomorphism $\Ftn(i(\ochi))\simeq i(\ochi)^\vee$ is given explicitly by
\[
\bfc\mapsto \lam_{\bfc}^{\ochi}:=\sum_{\gamma_i\in\oT\bs \oA} \bfc(\gamma_i)\lam_{\gamma_i}^{\ochi}.
\]

Consider the principal series $I(\ochi):=I(i(\ochi))$ for an unramified character $\ochi\in \Hom (Z(\oT),\bc^\times)$. We define a space of Whittaker functionals on $I(\ochi)$.

Let $\psi_U:U\to\bc^\times$ be the character on $U$ such that its restriction to every $U_\al,\al\in\Delta$ is given by $\psi\circ e_\al^{-1}$. We may write $\psi$ for $\psi_U$ if no confusion arises.

\begin{Def}
For any genuine representation $(\overline{\sigma}, V_{\overline{\sigma}})$ of $\oG$, a linear functional $\lam:V_{\overline{\sigma}}\to \bc$ is called a $\psi$-Whittaker functional if $\lam(\overline{\sigma}(u)v)=\psi(u)\cdot v$ for all $u\in U$ and $v\in V_{\overline{\sigma}}$. Write $\mathrm{Wh}_\psi(\overline{\sigma})$ for the space of $\psi$-Whittaker functionals for $\overline{\sigma}$.
\end{Def}

Consider the following integral
\[
\int_{U} f(w_G ug)\overline{\psi(u)} \ du
\]
for $f\in I(\ochi)$. This is a $i(\ochi)$-valued functional. To obtain a Whittaker functional, we need to apply an element in $i(\ochi)^\vee$. By \cite{McNamara16} Sect. 6, there is an isomorphism between $i(\ochi)^\vee$ and the space $\Wh_\psi(I(\ochi))$ of $\psi$-Whittaker functionals on $I(\ochi)$, given by $\lam\mapsto W_\lam$ with
\begin{equation}\label{eq:whittaker functional}
W_\lam:I(\ochi)\to \bc, \qquad f\mapsto \lam
\left( \int_U f(w_G u)\overline{\psi(u)} \ du \right),
\end{equation}
where $f\in I(\ochi)$ is an $i(\ochi)$-valued function on $\oG$; $w_G\in K$ is a representative of $\bw_G$.


For $\bfc\in \Ftn(i(\ochi))$, by abuse of notation, we will write $\lam_{\bfc}^{\ochi}\in \Wh_\psi (I(\ochi))$ for the resulting $\psi$-Whittaker functional of $I(\ochi)$ from the isomorphism $\Ftn(i(\ochi))\simeq i(\ochi)^\vee \simeq \Wh_{\psi}(I(\ochi))$. As a consequence, $\dim \Wh_\psi (I(\ochi)) \simeq |Y/Y_{Q,n}|$.



\subsection{Local coefficient matrix}
Let $J(\bw,\ochi)$ be the image of $T_{\bw,\ochi}$. The operator $T_{\bw,\ochi}$ induces a homomorphism $T_{\bw,\ochi}^{\ast}$ of vector spaces with image $\Wh_{\psi}(J(\bw,\chi))$:
\[
T_{\bw,\ochi}^{\ast}:\Wh_{\psi}(I({}^{\bw}\ochi))\to \Wh_{\psi}(I(\ochi))
\]
which is given by
\[
\la \lambda_{\bfc}^{{}^{\bw}\ochi}, - \ra \mapsto \la \lambda_{\bfc}^{{}^{\bw}\ochi}, T_{\bw,\ochi}(-) \ra
\]
for any $\bfc\in\Ftn(i({}^{\bw}\ochi))$. Let $\{\lambda_{\gamma}^{{}^{\bw}\ochi}\}_{\gamma\in \oT/\oA}$ be a basis for $\Wh_{\psi}(I({}^{\bw}\ochi))$, and $\lam_{\gamma'}^{\ochi}$ a basis for $\Wh_{\psi}(I(\ochi))$. The map $T_{\bw,\ochi}^{\ast}$ is then determined by the square matrix $[\tau(\ochi,\sr,\gamma,\gamma')]_{\gamma,\gamma'\in\oT/\oA}$ of size $|Y/Y_{Q,n}|$ such that
\[
T_{\bw,\ochi}^{\ast}(\lam_{\gamma}^{{}^{\bw}\ochi}) =\sum_{\gamma'\in\oT/\oA}\tau(\bw,\ochi,\gamma,\gamma') \lam_{\gamma'}^{\ochi}.
\]

The local coefficient matrix satisfies the following properties.

\begin{Lem}[\cite{KP84, McNamara16, Gao17}]\label{lem：local coefficient matrix cocycle}
For $\bw\in W$ and $\bar{z},\bar{z}'\in \bar{A}$, the following identity holds:
\[
\tau(\bw,\ochi,\gamma\cdot \oz,\gamma'\cdot \oz')
=({}^{\bw}\ochi^{-1}(\bar{z}))\cdot \tau( \bw,\ochi, \gamma,\gamma')\cdot \ochi(\bar{z}').
\]
Moreover, for $\bw_1,\bw_2\in W$ such that $\ell(\bw_1\bw_2)=\ell(\bw_1)+\ell(\bw_2)$, one has
\[
\tau( \bw_1\bw_2,\ochi, \gamma,\gamma')=\sum_{\gamma''\in\oT/\oA} \tau(\bw_1,{}^{\bw_2}\ochi,  \gamma,\gamma'')\cdot \tau(\bw_2,\ochi,  \gamma'',\gamma'),
\]
which is referred to as the cocycle relation.
\end{Lem}

\begin{proof}
This fact is standard. For example, it follows from \cite{Gao17} Lemma 3.2. 
\end{proof}

Thus the calculation of the local coefficient matrix $[\tau(\ochi,\bw,\gamma,\gamma')]_{\gamma,\gamma'}$ is reduced to the case when $\bw$ is a simple reflection.

We now would like to compute the matrix $[\tau(\ochi,\sr_\al,\gamma,\gamma')]_{\gamma,\gamma'}$  for any unramified character $\ochi$ and simple reflection $\sr_\al, \al\in\Delta$.

\begin{Thm}[\cite{KP84} Lemma I.3.3 ,\cite{McNamara16} Theorem 13.1, \cite{Gao17} Theorem 3.6.]\label{thm:local coefficient matrix}
Suppose that $\gamma=\bfs_{y_1}$ and $\gamma'=\bfs_y$ by $y$. Then we can write
\[
\tau(\sr_\al,\ochi,\gamma,\gamma') =\tau^1(\sr_\al,\ochi,\gamma,\gamma')+\tau^2(\sr_\al,\ochi,\gamma,\gamma')
\]
with the following properties:
\begin{itemize}
\item $\tau^i(\sr_\al,\ochi,\gamma\cdot \bar{z},\gamma'\cdot\bar{z}') =({}^{\sr_\al}\ochi)^{-1}(\bar{z})\cdot\tau^i(\sr_\al,\ochi,\gamma,\gamma')\cdot \ochi(\bar{z}'),\qquad \bar{z},\bar{z}'\in\bar{A}$;
\item $\tau^1(\sr_\al,\ochi,\gamma,\gamma')=0$ unless $y_1\equiv y \mod Y_{Q,n}$;
\item $\tau^2(\sr_\al,\ochi,\gamma,\gamma')$ unless $y_1\equiv \sr_\al[y]\mod Y_{Q,n}$.
\end{itemize}
Moreover,
\begin{itemize}
\item If $y_1=y$, then
\[
\tau^1(\sr_\al,\ochi,\gamma,\gamma')
=(1-q^{-1})\dfrac{\ochi(\bar{h}_\al(\varpi^{n_\al}))^{k_{y,\al}}}{1-\ochi(\oh_\al(\varpi^{n_\al}))}, \text{ where }k_{y,\al}=\left\lceil\dfrac{\la y,\al\ra}{n_\al}\right\rceil
\]
\item If $y_1=\sr_\al[y]$, then
\[
\tau^2(\sr_\al,\ochi,\gamma,\gamma')=
\bfgg_{\psi^{-1}}(\la y_\rho,\al\ra Q(\al^\vee)).
\]
\end{itemize}
\end{Thm}


\subsection{Explicit calculation of the local coefficient matrix}

Lemma \ref{lem：local coefficient matrix cocycle} and Theorem \ref{thm:local coefficient matrix} determine the local coefficient matrix completely. However, it is too complicated to obtain a general formula as one has to analyze the sum over $\oT/\oA$ inductively. In this section, we highlight some observations that will be useful for our calculation.

Notations: for $y,y'\in Y$, we write
\[
\tau(\bw,\ochi,y,y'):=\tau(\bw,\ochi,\bfs_{y},\bfs_{y'}).
\]

Let $\bw=\bw_1\cdots \bw_k$ be a reduced decompositioin of $\bw$ by simple reflections.

\begin{Lem}\label{lem:calcualation of tau 1}
The coefficient $\tau(\bw,\ochi,y,y')=0$ unless $y'\equiv \bw_k^{a_k}\cdots \bw_1^{a_1}[y] \mod Y_{Q,n}$ for some $a_1,\cdots, a_k\in\{0,1\}$.
\end{Lem}

\begin{proof}
We can prove this by induction on $k$. When $k=1$, this follows from Theorem \ref{thm:local coefficient matrix}. We now assume that the result is true for $k-1$. Then
\[
\tau(\bw,\ochi,y,y')=\sum_{y''\in Y/Y_{Q,n}} \tau(\bw_1\cdots \bw_{k-1},{}^{\bw_k}\ochi,  y,y'')\cdot \tau(\bw_k,\ochi,  y'',y').
\]
If this is nonzero, then $\tau(\bw_1\cdots \bw_{k-1},{}^{\bw_k}\ochi,  y,y'')\neq 0$ and $\tau(\bw_k,\ochi,  y'',y')\neq 0$ for some $y''$. This implies that
\[
y''\equiv \bw_{k-1}^{a_{k-1}}\cdots \bw_1^{a_1}[y] \mod Y_{Q,n}, \text{ for some }a_1,\cdots, a_{k-1}\in \{0,1\},
\]
and $y'\equiv \bw_{k}^{a_k}[y''] \mod Y_{Q,n}$ for some $a_k\in\{0,1\}$.
This proves the result.
\end{proof}

We have an immediate corollary.

\begin{Cor}
The coefficient $\tau(\bw_G,\ochi,y,y')=0$ unless $\bw[y]\equiv y' \mod Y_{Q,n}$ for some $\bw\in W$.
\end{Cor}






The next result is very useful for calculation.

\begin{Lem}\label{lem:tau unique decomposition}
Assume that $\bw=\bw_1\cdots \bw_k$ is a reduced decomposition of $\bw$, and for any two subexpressions $\bw_1^{a_1}\cdots \bw_{k}^{a_k}=\bw_1^{a'_1}\cdots \bw_{k}^{a'_k}$, $a_1,\cdots, a_k, a'_1,\cdots,a'_k\in \{0,1\}$, we have $a_i=a_i'$ for $i=1,\cdots, k$. If the orbit of $y$ is free, then
\[
\begin{aligned}
&\tau(\bw,\ochi,y,\bw_k^{a_k}\cdots \bw_{1}^{a_1}[y])\\
=&\tau(\bw_1,{}^{\bw_2\cdots \bw_k}\ochi,y,\bw_1^{a_1}[y]) \tau(\bw_2,{}^{\bw_3\cdots \bw_k}\ochi,\bw_1[y],\bw_2^{a_2}\bw_1^{a_1}[y])\\
&\cdots  \tau(\bw_k,\ochi,\bw_{k-1}^{a_{k-1}}\cdots \bw_{1}^{a_1}[y],\bw_k^{a_k}\cdots \bw_{1}^{a_1}[y]).
\end{aligned}
\]
In other words, only one term in the summation is nonzero.
\end{Lem}

\begin{proof}

The assumption implies that $\bw_k^{a_k}\cdots \bw_{1}^{a_1}[y]$ are all distinct in $Y/Y_{Q,n}$ for $a_1,\cdots,a_k\in \{0,1\}$.

We prove it by induction on $k$. If $k=1$, there is nothing to prove. Assume the result is true for $\bw_1\cdots \bw_{k-1}$. Then
\[
\tau(\ochi,\bw,y,\bw_k^{a_k}\cdots \bw_{1}^{a_1}[y])=\sum_{y''\in Y/Y_{Q,n}} \tau({}^{\bw_k}\ochi, \bw_1\cdots \bw_{k-1}, y,y'')\cdot \tau(\ochi, \bw_k, y'',\bw_k^{a_k}\cdots \bw_{1}^{a_1}[y]).
\]
For a nonzero term in the summation, we have
\[
y''\equiv \bw_{k-1}^{a'_{k-1}}\cdots \bw_{1}^{a'_1}[y] \mod Y_{Q,n} \text{ for some }a'_1,\cdots,a'_{k-1}\in \{0,1\}
\]
and $\bw_k^{a_k}\cdots \bw_{1}^{a_1}[y]\equiv \bw_k^{a'_k}[y''] \mod Y_{Q,n}$ for some $a_k\in \{0,1\}$. As the orbit of $y$ is free, this implies that
\[
\bw_k^{a_k}\cdots \bw_{1}^{a_1} =\bw_k^{a'_k}\cdots \bw_{1}^{a'_1}
\]
and therefore $a_i=a'_i$ for $i=1,\cdots,k$. We now conclude that only the term $y''=\bw_k^{a_k}\cdots \bw_{1}^{a_1}[y]$ has nonzero contribution in the summation and therefore
\[
\begin{aligned}
&\tau(\ochi,\bw,y,\bw_k^{a_k}\cdots \bw_{1}^{a_1}[y])\\
=&\tau({}^{\bw_k}\ochi, \bw_1\cdots \bw_{k-1}, y,\bw_{k-1}^{a_{k-1}}\cdots \bw_{1}^{a_1}[y])\cdot \tau(\ochi, \bw_k, \bw_{k-1}^{a_{k-1}}\cdots \bw_{1}^{a_1}[y]),\bw_k^{a_k}\cdots \bw_{1}^{a_1}[y]).
\end{aligned}
\]
By induction we obtain the desired formula.
\end{proof}

\begin{Rem}
The conditions in the lemma are satisfied in the following example: $\bg=\bgl_r$ and $\bw=\sr_{\al_1}\cdots \sr_{\al_r}$. We will use it later.
\end{Rem}

Notice that $Y_{Q,n}$ is not well-behaved with respect to Levi subgroup so it is better to work with the lattice $Y_{Q,n}^{sc}$. Observe that  $Y_{Q,n}\cap Y^{sc}=Y_{Q,n}^{sc}$.

\begin{Lem}\label{lem:several lattices}
If $y\equiv \bw[y]\mod Y_{Q,n}$ for some $\bw \in W$, then $y\equiv \bw[y]\mod Y_{Q,n}^{sc}$.
\end{Lem}

\begin{proof}
By \cite{BBF08} Lemma 2, $y-\bw[y]\in Y^{sc}$. If $y-\bw[y]\in Y_{Q,n}$, then it is in $Y_{Q,n}^{sc}$.
\end{proof}

Let $\bt_{Q,n}^{sc}$ be the split torus with cocharacter group $Y_{Q,n}^{sc}$, and $T_{Q,n}^{sc}:=\bt_{Q,n}^{sc}(F)$.

\begin{Lem}\label{lem:dependence on chi}
The coefficient $\tau(\bw,\ochi,y,\bw'[y])$ depends only on $\ochi|_{\oT_{Q,n}^{sc}}$ for $\bw,\bw'\in W$.
\end{Lem}

\begin{proof}
If $\bw=\id$, this result follows from Lemma \ref{lem:several lattices} and Lemma \ref{lem：local coefficient matrix cocycle}.

We now consider a nontrivial Weyl group element with reduced decomposition $\bw=\bw_1\cdots \bw_k$. If $\tau(\bw,\ochi,y,\bw'[y])\neq 0$, then by Lemma \ref{lem:calcualation of tau 1}, $\bw'[y]\equiv \bw_k^{a_k}\cdots \bw_1^{a_1}[y] \mod Y_{Q,n}$ for some $a_1,\cdots, a_k\in\{0,1\}$. By Lemma \ref{lem:several lattices}, $\bw'[y]\equiv \bw_k^{a_k}\cdots \bw_1^{a_1}[y] \mod Y_{Q,n}^{sc}$. So it suffices to prove the result for elements of the form $\bw'=\bw_k^{a_k}\cdots \bw_1^{a_1}$.

We now argue by induction on the length of $\bw$. If $\bw=\sr_\al$, then the result is straightforward when $y\not\equiv \sr_\al[y]\mod Y_{Q,n}$. If $y\equiv \sr_\al[y]\mod Y_{Q,n}$, then $y\equiv \sr_\al[y]\mod Y_{Q,n}^{sc}$. The same argument above applies. The same argument again applies in the induction argument. This proves the result.
\end{proof}

\subsection{Unramified Whittaker functions}
For an unramified principal series representation $I(\ochi)$, let $\cw$ be the image of $\phi_K$ in the Whittaker model defined by \eqref{eq:whittaker functional}. In other words,
\[
\cw_\lam(\ot,\ochi)=\delta_B^{-1/2}(\ot)W_{\lam}(\ot\cdot \phi_K).
\]
Note that our definition here is slightly different from \cite{Gao18b}. We divide by the modular quasi-character $\delta_B^{-1/2}$ to make our calculation slightly easier. If $\lam$ is defined by $\gamma$, we write $\cw_\gamma=\cw_{\lam_\gamma}$.
We also define $\cw_{y}(y',\ochi)=\cw_{\bfs_y}(\bfs_{y'},\ochi)$.

An element $\ot\in\oT$ is called dominant if $\ot\cdot (U\cap K)\cdot \ot^{-1}\subset K$.
\begin{Thm}\label{thm:unramified whittaker function}
Let $I(\ochi)$ be an unramified principal series of $\oG$ and $\gamma\in\oT$. Let $\cw_\gamma$ be the unramified Whittaker function associated to $\phi_K$. Then, $\cw_\gamma(\ot)=0$ unless $\ot\in\oT$ is dominant. Moreover, for dominant $\ot$, one has
\[
\cw_\gamma(\ot,\ochi)=\sum_{\bw\in W} c(\bw_G \bw,\ochi) \cdot \tau(\bw^{-1},{}^{\bw}\ochi,\gamma,w_G\cdot \ot\cdot w_{G}^{-1}).
\]
\end{Thm}

\begin{proof}
The proof in \cite{Gao18b} Proposition 3.3 works without essential change.
\end{proof}

\section{An inductive formula}\label{sec:inductive}

As a consequence of Theorem \ref{thm:unramified whittaker function}, we now prove an inductive formula for unramified Whittaker function. The main result in this section is a generalization of the material presented in \cite{Suzuki98} Section 7.1.

For certain types of root systems, our formula might admit simplification -- we discuss this in Sect. \ref{sec:general linear}. See also \cite{Suzuki97} Lemma 4.1 and \cite{Suzuki98} Section 7.1. Note that there are some typos in the proofs of these two papers. We give full details here.

\subsection{Basic setup}

Let $\Delta'$ be a subset of $\Delta$. Let $\bp=\bm\bn$ be the parabolic subgroup of $\bg$ associated with $\Delta'$. We write
\[
(X, \ \Phi_M, \ \Delta_M; \  Y, \ \Phi_M^\vee, \ \Delta_M^\vee)
\]
for the root datum of $M$. Since $\bt\subset \bm$, the character and cocharacter lattices $X$ and $Y$ respectively are unchanged. However, we have $\Delta_M= \Delta'$ and $\Delta_M^\vee =\{\beta^\vee:\beta\in\Delta'\}$. Let $\bb_M=\bt \bu_M$ be the Borel subgroup of $\bm$ corresponding to $\Delta_M$. Denote by $W(\bm)\subset W(\bg)$ the Weyl group of $(\bm,\bt)$.

The functorial properties with respect to restriction is studied in \cite{GG18} Sect. 5.5. The cover $\overline{M}$ is associated to the pair $(D,\eta|_{Y_M^{\mathrm{sc}}})$, where the quadratic form $Q(x)=D(x,x)$ carries only the $W(\bm)$-invariance by applying the ``forgetful'' functor from $W$-invariance.

Given a genuine character $\ochi:\oT\to\bc^\times$, one can define an unramified principal series representation $I_{\oM}(\ochi)$ on $\oM$. By induction in stages, $I(\ochi)=\Ind_{\oP}^{\oG}I_{\oM}(\chi)$. Here $I_{\oM}(\ochi)$ is inflated to a representation on $\oP$ in the usual way. The study of Whittaker models and Whittaker functions applies to representations on $\oM$. We add subscript $\oM$ to indicate the ambient group.

We have the following observations:
\begin{itemize}
\item The section $\bfs_{\oM,y}=\bfs_{\oG,y}$ for $y\in Y$. So the notation $\bfs_y$ does not arise any confusion.
\item For $\bw\in W(\bm)$, one can calculate the local coefficient matrix $\tau_{\oM}(\bw,\ochi,y,y')$. It is easy to check that $\tau_{\oM}(\bw,\ochi,y,y')=\tau_{\oG}(\bw,\ochi,y,y')$. Thus we can safely drop the subscript.
\end{itemize}

Let $W^M$ be the set of minimal representatives in $W(\bm)\bs W$.
A element $\bw\in W$ can be uniquely written as $\bw=\bw_1\bw_2$, where $\bw_1\in W(\bm)$ and $\bw_2\in W^M$. The long element $\bw_G$ is written as $\bw_M\bw^M$.

\begin{Lem}
We have
\[
\Phi(\bw_G \bw)=\Phi(\bw^{M,-1}\bw_2)\sqcup \bw_2^{-1}(\Phi_M(\bw_M \bw_1)).
\]
\end{Lem}

\begin{proof}
Observe that
\[
\Phi(\bw_G\bw)=\{\al>0: \bw(\al)>0\}.
\]
and any element in this set satisfies $\bw_2(\al)>0$.
We have
\[
\{\al>0: \bw(\al)>0\}=\{\al>0: \bw_2(\al)\in \Phi^+-\Phi_M^+,\bw(\al)>0\}\sqcup \{\al>0: \bw_2(\al)\in\Phi_M^+,\bw(\al)>0\}.
\]
We now show that the first set is $\Phi(\bw^{M,-1}\bw_2)$ and the second set is $\bw_2^{-1}(\Phi_M(\bw_M \bw_1))$.

Note that $\Phi(\bw^{M,-1})=\Phi^+-\Phi_M^+$. Thus
\begin{equation}\label{eq:the first set}
\Phi(\bw^{M,-1}\bw_2)=\{\al>0:\bw^{M,-1}\bw_2(\al)<0\}=\{\al>0:\bw_2(\al)\in \Phi^+ - \Phi_M^+\}.
\end{equation}
Note that if $\bw_2(\al) \in \Phi^+ - \Phi_M^+$, then $\bw(\al)>0$. Thus \eqref{eq:the first set} is the first set.

Let $\beta=\bw_2(\al)$. Then the second set is
\[
\{\al>0: \bw_2(\al)\in\Phi_M^+,\bw(\al)>0\}=\{\bw_2^{-1}(\beta)\in \Phi_M^+:\bw_1(\beta)>0\}=\bw_2^{-1}(\Phi_M(\bw_M \bw_1)).
\]
Now the result follows.
\end{proof}

\subsection{The inductive formula}

We now give the inductive formula.
\begin{Prop}\label{prop:inductive formula}
We have
\[
\begin{aligned}
\cw_\gamma(\ot,\ochi)=&\sum_{\bw_2\in W^M}\sum_{\gamma'\in \oT/\oA} \tau(\bw_2^{-1},{}^{\bw_2}\ochi,\gamma,\gamma') \cw_{M,\gamma'}(w^{M}\cdot \ot \cdot w^{M,-1},{}^{\bw_2}\ochi)\\
&\cdot \left(\prod_{\al>0:\bw^{M,-1}\bw_2(\al)<0} \dfrac{1-\ochi_\al q^{-1}}{1-\ochi_\al} \right).
\end{aligned}
\]
\end{Prop}

\begin{proof}
Recall that
\[
\cw_\gamma(\ot,\ochi)=\sum_{\bw\in W} c(\bw_G \bw,\chi) \cdot \tau(\bw^{-1},{}^{\bw}\chi,\gamma,w_G\cdot \ot\cdot w_{G}^{-1})
\]
Given $\bw\in W$, it can be uniquely written as $\bw=\bw_1\bw_2$ as above. By the cocycle relation in Lemma \ref{lem：local coefficient matrix cocycle}, we deduce that
\[
\tau(\bw^{-1},{}^{\bw}\ochi,\gamma,w_G\cdot \bar{t}\cdot w_{G}^{-1})
=
\sum_{\gamma'\in \oT/\oA} \tau(\bw_2^{-1},{}^{\bw_2}\ochi,\gamma,\gamma')
\tau(\bw_1^{-1},{}^{\bw}\chi,\gamma',w_G\cdot \bar{t}\cdot w_{G}^{-1})
\]
On the other hand,
\[
\begin{aligned}
c(\bw_G\bw,\ochi)=&\prod_{\al>0, \ \bw_G\bw(\al)<0}\dfrac{1-\ochi_\al q^{-1}}{1-\ochi_\al}\\
=&\left(\prod_{\al>0:\bw^{M,-1}\bw_2(\al)<0} \dfrac{1-\ochi_\al q^{-1}}{1-\ochi_\al} \right)
\left(\prod_{\bw_2^{-1} \{ \al>0:\bw_M\bw_1(\al)<0 \}}\dfrac{1-\ochi_{\bw_2^{-1}(\al)} q^{-1}}{1-\ochi_{\bw_2^{-1}(\al)}}\right)\\
=& \left(\prod_{\al>0:\bw^{M,-1}\bw_2(\al)<0} \dfrac{1-\ochi_\al q^{-1}}{1-\ochi_\al} \right)
\left(\prod_{\al>0:\bw_M\bw_1(\al)<0 }\dfrac{1-({}^{\bw_2}\ochi)_\al q^{-1}}{1-({}^{\bw_2}\ochi)_\al}\right).
\end{aligned}
\]
Here, we use the following fact: $\ochi_{\bw^{-1}(\al)}=({}^{\bw}\ochi)_\al$.
This can be seen from the following identity: $\la x,\bw^{-1}(\al^\vee)\ra =\la \bw(x),\al^\vee\ra$ for any $x\in X$.

From this we deduce that
\[
\begin{aligned}
&\sum_{\bw\in W} c(\bw_G \bw,\chi) \cdot \tau(\bw^{-1},{}^{\bw}\ochi,\gamma,w_G\cdot \ot\cdot w_{G}^{-1})\\
=&\sum_{\bw_2\in W^M}\sum_{\bw_1\in W(\bm)}\sum_{\gamma'\in \oT/\oA} \tau(\bw_2^{-1},{}^{\bw_2}\ochi,\gamma,\gamma')
\tau(\bw_1^{-1},{}^{\bw}\chi,\gamma',w_G\cdot \bar{t}\cdot w_{G}^{-1})\\
 &\cdot \left(\prod_{\al>0:\bw^{M,-1}\bw_2(\al)<0} \dfrac{1-\ochi_\al q^{-1}}{1-\ochi_\al} \right)
\left(\prod_{\al>0:\bw_M\bw_1(\al)<0 }\dfrac{1-({}^{\bw_2}\ochi)_\al q^{-1}}{1-({}^{\bw_2}\ochi)_\al}\right).\\
\end{aligned}
\]
Note that
\[
\sum_{\bw_1\in W(\bm)} \tau(\bw_1^{-1},{}^{\bw}\chi,\gamma',w_G\cdot \bar{t}\cdot w_{G}^{-1})
\left(\prod_{\al>0:\bw_M\bw_1(\al)<0 }\dfrac{1-({}^{\bw_2}\ochi)_\al q^{-1}}{1-({}^{\bw_2}\ochi)_\al}\right)
=\cw_{M,\gamma'}(w^{M}\cdot \ot\cdot  w^{M,-1},{}^{\bw_2}\ochi).
\]
Thus we deduce that $\cw_\gamma(\ot,\ochi)$ equals
\[
\sum_{\bw_2\in W^M}\sum_{\gamma'\in \oT/\oA} \tau(\bw_2^{-1},{}^{\bw_2}\ochi,\gamma,\gamma') \cw_{M,\gamma'}(w^{M}\cdot \ot\cdot w^{M,-1},{}^{\bw_2}\ochi)
\left(\prod_{\al>0:\bw^{M,-1}\bw_2(\al)<0} \dfrac{1-\ochi_\al q^{-1}}{1-\ochi_\al} \right).
\]
\end{proof}

\subsection{Local coefficient matrix}\label{sec:local mat coeff for Levi}
We end this section with a useful result on the local coefficient matrix.
We now write $\bm=\bm_1\times \cdots \times \bm_k$. Let $\bt_i=\bt\cap \bm_i$. Let $Y_i$ be the cocharacter lattice of $\bt_i$. Let $W(\bm_i)$ be the Weyl group of $(\bm_i,\bt_i)$.

Let $\bw=(\bw_1,\cdots,\bw_k)\in W(\bm)$ with $\bw_i\in W(\bm_i)$. Let $\bw'=(\bw'_1,\cdots,\bw'_k)\in W(\bm)$ with $\bw'_i\in W(\bm_i)$. Let $y=(y_1,\cdots,y_k)$ where $y_i\in Y_i$. Let $y'_i=\bw_i'[y_i]$.

We now consider $\ochi|_{\oT_{Q,n}^{sc}}$. 
Let $\ochi_i$ be a character of $Z(\oT_i)$ so that its restriction to $\oT_{i,Q,n}^{sc}$ agrees with $\ochi_i|_{\oT_{i,Q,n}^{sc}}$. In such situations, we write $\ochi\sim (\ochi_1,\cdots,\ochi_k)$. Recall from Lemma \ref{lem:dependence on chi} that $\tau(\ochi_i,\bw_i,y,\bw_i'[y'])$ only depends on the choice of $\ochi_i|_{\oT_{i,Q,n}^{sc}}$ but not on the choice of $\ochi_i$.

\begin{Lem}\label{lem:10}
With notations as above,
\[
\tau(\bw,\ochi,y,y')=\prod_{i=1}^k \tau(\bw_i,\ochi_i,y_i,y'_i).
\]
\end{Lem}

\begin{proof}
By induction, it suffices to prove the case $k=2$. So we assume $k=2$ from now on.

For the case $k=2$, we prove it by induction on the length of $\bw$. If $\bw=\id$, the result is trivial.

We now assume the result is true for $\bw$ and prove it for $\sr_\al \bw$ where $\ell(\sr_\al \bw)=\ell(\bw)+1$ and $\sr_\al$ is in either $W(\bm_1)$ or $W(\bm_2)$. We assume that $\sr_\al \in W(\bm_1)$ without loss of generality. We have
\[
\tau(\ochi,\sr_\al\bw,y,\bw'[y])
=\sum_{y''\in Y/Y_{Q,n}} \tau({}^{\bw}\ochi,\sr_\al,y,y'')\tau(\ochi,\bw,y'',\bw'[y]).
\]
The first term is nonzero only when $y''=y$ or $\sigma_\al[y]$. We write $y''=(y_1'',y_2'')$. By induction, we have
\[
\tau(\ochi,\bw,y'',\bw'[y])=\tau(\ochi_1,\bw_1,y''_1,\bw'_1[y_1]) \tau(\ochi_2,\bw_2,y''_2,\bw'_2[y_2]).
\]

Note that $\sr_\al[y]=(\sr_\al[y_1],y_2)$ and $\tau({}^{\sr_\al}\ochi,\sr_\al,y,\sr_\al[y])=\tau({}^{\sr_\al}\ochi_1,\sr_\al ,y_1,\sr_\al[y_1])$. By Lemma \ref{lem:several lattices}, it is easy to verify that $y\equiv\sr_\al[y]\mod Y_{Q,n}$ if and only if $y_1\equiv \sr_\al[y_1] \mod Y_{1,Q,n}$. If $y\not\equiv\sr_\al[y] \mod Y_{Q,n}$, then
\[
\begin{aligned}
&\tau(\ochi,\sr_\al \bw,y,y')\\
=&\tau(\sr_\al ,{}^{\bw}\ochi,y,y)\tau(\ochi_1,\bw_1,y_1,y'_1)\tau(\ochi_2,\bw_2,y_2,y'_2)
+\tau(\sr_\al ,{}^{\bw}\ochi,y,\sr_\al[y])\tau(\ochi_1,\bw_1,\sr_\al[y_1],y'_1)\tau(\ochi_2,\bw_2,y_2,y'_2)\\
=&(\tau(\sr_\al ,{}^{\bw}\ochi,y,y)\tau(\ochi_1,\bw_1,y_1,y'_1)
+\tau(\sr_\al ,{}^{\bw}\ochi,y,\sr_\al[y])\tau(\ochi_1,\bw_1,\sr_\al[y_1],y'_1))\tau(\ochi_2,\bw_2,y_2,y'_2)\\
=&\tau(\ochi_1,\sr_\al \bw_1,y_1,y_1')\tau(\ochi_2,\bw_2,y_2,y'_2)\\
\end{aligned}
\]
If $y\equiv \sr_\al[y] \mod Y_{Q,n}$,
\[
\begin{aligned}
&\tau(\ochi,\sr_\al \bw,y,y')\\
=&\tau(\sr_\al ,{}^{\bw}\ochi,y,y)\tau(\ochi_1,\bw_1,y_1,y'_1)\tau(\ochi_2,\bw_2,y_2,y'_2)\\
=&\tau(\ochi_1,\sr_\al \bw_1,y_1,y_1')\tau(\ochi_2,\bw_2,y_2,y'_2).\\
\end{aligned}
\]
\end{proof}

\section{Relative theta representations}\label{sec:relative theta}

We first recall the definition of theta representations and discuss its generalization given in \cite{Suzuki98} and \cite{Gao}.

\subsection{Definition}

We start with the following definition.

\begin{Def}
An unramified genuine character $\ochi$ of $Z(\oT)$ is called \textit{exceptional} if
\[
\ochi(\oh_\al(\varpi^{n_\al}))=q^{-1} \text{ for all }\al\in\Delta.
\]
The theta representation $\Theta(\oG,\ochi)$ associated to an exceptional character $\ochi$ is the unique Langlands quotient (see \cite{BJ13}) of $I(\ochi)$, which is also equal to the image of the intertwining operator $T_{\bw_G,\ochi}:I(\ochi)\to I({}^{\bw_G}\ochi)$.
\end{Def}

To make our discussion more flexible, we introduce the following definition. It can be viewed as a generalization of \cite{Suzuki98} and \cite{Gao}.
\begin{Def}\label{def:exceptional}
For any subset $\Delta'\subset \Delta$, a genuine character $\chi$ is called $\Delta'$-exceptional (resp. $\Delta'$-anti-exceptional) if $\ochi(\oh_\al(\varpi^{n_\al}))=q^{-1}$ (resp. $\ochi(\oh_\al(\varpi^{n_\al}))=q$) for every $\al\in\Delta'$. In the case $\Delta'=\Delta$, it is simply called exceptional or anti-exceptional, respectively.
\end{Def}

Let $\bm$ be the Levi subgroup corresponding to $\Delta'$. Then a $\Delta'$-exceptional character can be viewed as an exceptional character for $\oM$. In other words, we obtain a representation $\Theta(\oM,\ochi)$ of $\oM$ as the image of the intertwining operator
\[
T_{\bw_{M},\ochi}:I_{\oM}(\ochi)\to I_{\oM}({}^{\bw_M}\ochi).
\]
Here $\bw_M$ is the longest element in the Weyl group of $M$. We also add subscript `$\oM$' to indicate the ambient group. We will do so in the rest of this section.


We can now define a representation on $\oG$ by normalized induction:
\[
\Theta(\oG/\oM,\ochi):=\Ind_{\oP}^{\oG}\Theta(\oM,\ochi).
\]
We call it a \textit{relative Theta representation}.
The representation $\Theta(\oG/\oM,\ochi)$ can also be defined as the image of the intertwining operator
\[
T_{\bw_M,\ochi}:I(\ochi)\to I({}^{\bw_{M}}\ochi).
\]
Note that $\Theta(\oG/\oM,\ochi)$ might be reducible.

\subsection{Some properties}

We discuss some properties of $\Theta(\oG/\oM,\ochi)$.
The intertwining operator $T_{\bw_M,\ochi}:I(\ochi)\to I({}^{\bw_{M}}\ochi)$ induces a map on the space of Whittaker functional
\[
T_{\bw_M,\ochi}^{\ast}: \Wh_{\psi}(I({}^{\bw_{M}}\ochi))\to \Wh_{\psi}(I(\ochi)).
\]
The matrix is defined by
\[
T_{\bw_M,\ochi}^{\ast}(\lambda_{\gamma}^{{}^{\bw_M}\ochi})=\sum_{\gamma'\in \oT/\oA} \tau(\bw_M,\ochi,\gamma,\gamma')\cdot \lambda_{\gamma'}^{\ochi}.
\]

\begin{Prop}
A function $\bfc \in \Ftn(i(\ochi))$ gives rise to a functional in $\Wh_{\psi}(\Theta(\oG/\oM,\ochi))$ if and only if for all $\al\in \Delta'$,
\[
\sum_{\gamma\in\oT/\oA} \bfc(\gamma) \tau(\sr_\al,{}^{\sr_\al}\ochi,\gamma,\gamma')=0 \text{ for all }\gamma'.
\]
The left-hand side is independent of the choice of representatives for $\oT/\oA$.
\end{Prop}

\begin{proof}
The same proof in \cite{KP84} page 76 works here as well.
\end{proof}

\begin{Prop}
Let $\ochi$ be an unramified $\Delta'$-exceptional character. Let $\lambda_{\bfc}^{\ochi}\in \Wh_\psi(I(\ochi))$ be the $\psi$-Whittaker functional of $I(\ochi)$ associated to some $\bfc\in \Ftn(i(\ochi))$. Then, $\lambda_{\bfc}^{\ochi}$ lies in $\Wh_{\psi}(\Theta(\oG/\oM,\ochi))$ if and only for any simple root $\al\in \Delta'$ one has
\[
\bfc(\bfs_{\sr_\al[y]})=q^{k_{y,\al}-1}
\bfgg_{\psi^{-1}}(\la y,\al^\vee\ra Q(\al^\vee))^{-1}\cdot \bfc(\bfs_y) \text{ for all }y.
\]
\end{Prop}

\begin{proof}
The proof in \cite{Gao17} Corollary 3.7 works the same here.
\end{proof}

We now state some basic properties of these coefficients. See also \cite{Suzuki97} Sect. 3.6.

\begin{Prop}\label{prop:tau on weyl orbit}
Let $\ochi$ be an unramified $\Delta'$-exceptional character.
\begin{enumerate}
  \item $\tau(\bw_M,\ochi,y,y')=0$ unless $y'=\bw[y]$ for some $\bw \in W(\bm)$.
  \item $\tau(\bw_M,\ochi,y,\bw[y'])=R(\bw,y)\tau(\bw_M,\ochi,y,y')$ for $\bw\in W(\bm)$, where $R(\bw,y)$ is some function of $W$ and $Y$.
\end{enumerate}
\end{Prop}

\begin{proof}
The first one is obvious. The second one follows from the above lemma. In fact, $\tau(\ochi,\bw_M,\gamma,-)\in \Ftn(i(\ochi))$ gives rise to a functional in $\Wh_\psi(\Theta(\oG/\oM,\ochi))$. When $\bw=\sr_\al$ is a simple reflection, then
\[
\tau(\bw_M,\ochi,y,\sr_\al[y'])=q^{k_{y,\al}-1}\bfgg_{\psi^{-1}}(\la y,\al\ra Q(\al^\vee))^{-1}\tau(\bw_M,\ochi,y,y').
\]
The rest follows by induction.
\end{proof}

\begin{Cor}\label{cor:proportional}
Let $\bw \in W(\bm)$. Then $T_{\bw_M,\ochi}^{\ast}(\lambda_{y}^{{}^{\bw_M}\ochi})$ and $T_{\bw_M,\ochi}^{\ast}(\lambda_{\bw[y]}^{{}^{\bw_M}\ochi})$ are proportional on $I(\ochi)$, and
\[
T_{\bw_M,\ochi}^{\ast}(\lambda_{\bw[y]}^{{}^{\bw_M}\ochi}) =\dfrac{\tau(\bw_M,\ochi,\bw[y],y)}{\tau(\bw_M,\ochi,y,y)} T_{\bw_M,\ochi}^{\ast}(\lambda_{y}^{{}^{\bw_M}\ochi})
\]
\end{Cor}

\begin{proof}
This is an immediate consequence of Proposition \ref{prop:tau on weyl orbit}. In fact,
\[
T_{\bw_M,\ochi}^{\ast}(\lambda_{y}^{{}^{\bw_M}\ochi}) =\sum_{y'\in Y/Y_{Q,n}}\tau(\bw_M,\ochi,y,y') \lam_{y'}^{\ochi} =\sum_{\tilde \bw\in W(\bm) }R(\tilde\bw,y)\cdot \tau(\bw_M,\ochi,y,y) \lam_{\tilde\bw[y']}^{\ochi}
\]
and similarly
\[
T_{\bw_M,\ochi}^{\ast}(\lambda_{\bw[y]}^{{}^{\bw_M}\ochi}) =\sum_{y'\in Y/Y_{Q,n}}\tau(\bw_M,\ochi,\bw[y],y') \lam_{y'}^{\ochi}
=\sum_{\tilde \bw\in W(\bm) }R(\tilde\bw,y)\cdot \tau(\bw_M,\ochi,\bw[y],y) \lam_{\tilde\bw[y']}^{\ochi}.
\]
This gives the desired result.
\end{proof}

\subsection{Rodier's lemma}\label{sec:rodier}
We end this section with a generalization of a lemma of Rodier. This will be useful later. Recall that Rodier's result says that when the inducing data is generic, then so is the induced representation.

\begin{Prop}\label{prop:rodier}
The representation $\Ind_{\oP}^{\oG} (\pi)$ is generic if and only if $\pi$ is generic. Moreover,
\[
\dim \Hom_{U}(\Ind_{\oP}^{\oG} (\pi),\psi) =\dim \Hom_{U_M}(\pi,\psi|_{U_M}).
\]
\end{Prop}

\begin{proof}
This follows from \cite{BZ77} Theorem 5.2 and \cite{CS80} Lemma 1.5.
\end{proof}

\section{The case of general linear groups}\label{sec:general linear}

From now on, we focus on the case of $\bg=\bgl_{r+1}$. We now introduce some notations in this setup. Write $\Delta=\{\al_1,\cdots,\al_r\}$ with the standard enumeration and the Weyl group is generated by $\{\sr_1,\cdots,\sr_r\}$. The root system is simply-laced, and we write $n_Q=n_\al$ for any $\al\in \Delta$. For $\al=\al_i+\cdots+\al_{j-1}$, write $\ochi_{ij}=\ochi_\al$.

\subsection{Inductive formula}

The inductive formula in Proposition \ref{prop:inductive formula} admits a refinement in the case of $\bm=\bgl_{r}\times \bgl_1$ in $\bgl_{r+1}$.. This is similar to \cite{Suzuki97} Lemma 4.1. In this case, the $W^M$ is $\{ \sr_r,\ \sr_{r}\sr_{r-1},\ \cdots,\ \sr_r\cdots \sr_1 \}$.

Recall that $\cw_y(y',\ochi)\neq 0$ if and only if $y$ and $\bw_G(y')$ lies in the same orbit under the Weyl group action. (Note that this is not $\bw_G[y']$.)  We now assume that $\bw_G(y')=\bw'[y]$ for some $\bw'\in W$. Note that this is true identity instead of $\mod Y_{Q,n}$. Any $\bw'\in W$ can be uniquely written as $\bw'=\bw \sr_r\cdots \sr_{r_0}$ for an integer $r_0$ and $\bw \in W(\bm)$. We have arrived at
\[
\bw_G(y') = \bw \sr_r\cdots \sr_{r_0} [y]
\]
with $\bw\in W(\bm)$.

\begin{Prop}\label{prop:inductive 1 in general linear}
Assume that the orbit of $y$ under $W$ is free.
We have
\[
\cw_{y}(y',\ochi)=
\sum_{i=1}^{r_0}\sum_{y''} \tau(\sr_i\cdots\sr_r,{}^{\sr_r\cdots \sr_i}\ochi,y,y'') \cw_{M,y''}(\bw^{M}(y'),{}^{\sr_r\cdots \sr_i}\ochi) \left(\prod_{j=1}^{i}\dfrac{1-\ochi_{ji}q^{-1}}{1-\ochi_{ji}} \right),
\]
where the second sum is over the set
\[
\begin{aligned}
&\left\{
\sr_r\cdots \sr_{r_0} \sr_{r_0-2}^{a_{r_0-2}}\cdots \sr_i^{a_i}[y]: a_i,\cdots,a_{r_0-2}\in \{0,1\}
\right\}\\
=&\left\{
\sr_{r_0-2}^{a_{r_0-2}}\cdots \sr_i^{a_i}\bw^{-1}[\bw_G(y')]: a_i,\cdots,a_{r_0-2}\in \{0,1\}
\right\}.\\
\end{aligned}
\]
\end{Prop}

This result is probably true in general. But we only prove what we need here.

\begin{proof}
Note $\bw^M=\sr_r\cdots \sr_1$.
If $\bw_2=\sr_r\cdots \sr_i$, then $\bw^{M,-1}\bw_2=\sr_1\cdots \sr_{i-1}$.
Thus,
\[
\{\al>0:\bw^{M,-1}\bw_2(\al)<0\}=
\{
\al_{i-1}, \ \al_{i-2}+\al_{i-1},\cdots, \al_1+\cdots + \al_{i-1}
\}.
\]
In this case,
\[
\prod_{\al>0:\bw^{M,-1}\bw_2(\al)<0} \dfrac{1-\ochi_\al q^{-1}}{1-\ochi_\al} =\prod_{j=1}^{i}\dfrac{1-\ochi_{ji}q^{-1}}{1-\ochi_{ji}}.
\]
We may rewrite the formula as
\[
\begin{aligned}
&\cw_y(y',\ochi)\\
=&\sum_{i=1}^r\sum_{y''\in Y/Y_{Q,n}} \tau(\sr_i\cdots\sr_r,{}^{\sr_r\cdots \sr_i}\ochi,y,y'') \cw_{M,y''}(\bw^{M}(y'),{}^{\sr_r\cdots \sr_i}\ochi) \left(\prod_{j=1}^{i}\dfrac{1-\ochi_{ji}q^{-1}}{1-\ochi_{ji}} \right).
\end{aligned}
\]
We now analyze when both $\tau(\sr_i\cdots\sr_r,{}^{\sr_r\cdots \sr_i}\ochi,y,y'')$ and $\cw_{M,y''}(\bw^{M}(y'),{}^{\bw_2}\ochi)$ are nonzero.

We know that $\bw_G(y')=\bw \sr_r\cdots \sr_{r_0} [y]$ with $\bw\in W(\bm)$. If $\cw_{M,y''}(\bw^{M}(y'),{}^{\bw_2}\ochi) \neq 0$, then for some $y''\in Y$ and $\bw' \in W$,
\[
y''\equiv \bw'[\bw_G(y')] \mod Y_{Q,n}.
\]
This condition implies that
\[
y''\equiv \bw''\sr_r\cdots \sr_{r_0} [y] \mod Y_{Q,n}
\]
for some $\bw''\in W(\bm)$. 

If $\tau(\sr_i\cdots\sr_r,{}^{\sr_r\cdots \sr_i}\ochi,y,y'')\neq 0$, then
\[
\sr_r^{a_r}\cdots \sr_i^{a_i}[y]\equiv y'' \mod Y_{Q,n}
\]
for some $a_i,\cdots,a_r\in \{0,1\}$.

We now use the assumption that the orbit of $y$ is free. This implies that
\begin{equation}\label{eq:equality of Weyl}
\sr_r^{a_r}\cdots \sr_i^{a_i}=\bw''\sr_r\cdots \sr_{r_0}
\end{equation}
for some $a_i,\cdots,a_r\in \{0,1\}$ and $\bw''\in W(\bm)$.
By considering the images of both sides in $W(\bm)\bs W$, we know that this is possible only when $1\leq i\leq r_0$. The same argument shows that we must have $a_{r_0-1}=0$. Thus we conclude that elements in \eqref{eq:equality of Weyl} is of the form
\[
\sr_r\cdots \sr_{r_0} \sr_{r_0-2}^{a_{r_0-2}}\cdots \sr_i^{a_i}.
\]
So finally, we have arrived at
\[
\cw_y(y',\ochi)=\sum_{i=1}^{r_0}\sum_{y''} \tau(\sr_i\cdots\sr_r,{}^{\sr_r\cdots \sr_i}\ochi,y,y'') \cw_{M,y''}(\bw^M(y'),{}^{\sr_r\cdots \sr_i}\ochi) \left(\prod_{j=1}^{i}\dfrac{1-\ochi_{ji}q^{-1}}{1-\ochi_{ji}} \right)
\]
where the second sum is over the set
\[
\begin{aligned}
&\left\{
\sr_r\cdots \sr_{r_0} \sr_{r_0-2}^{a_{r_0-2}}\cdots \sr_i^{a_i}[y]\mid a_i,\cdots,a_{r_0-2}\in \{0,1\}
\right\}\\
=&\left\{
\sr_{r_0-2}^{a_{r_0-2}}\cdots \sr_i^{a_i}\bw^{-1}[\bw_G(y')]\mid a_i,\cdots,a_{r_0-2}\in \{0,1\}
\right\}.\\
\end{aligned}
\]
\end{proof}

We now discuss the covering group obtained by
\[
\bgl_r\hookrightarrow \bgl_{r+1},\qquad g\mapsto \mathrm{diag}(g,1).
\]
Write $Y=Y_1\oplus Y_2$ where $Y_1$ (resp. $Y_2$) be the cocharacter lattice of $\bgl_r$ (resp. $\bgl_1$). Then we have embeddings $\iota:Y_1\hookrightarrow Y$ and $\iota:Y_1^{sc}\hookrightarrow Y^{sc}$. The cover of $\overline{\GL}_r$ is associated with $(D|_{Y'},\eta|_{Y_1^{sc}})$. We also observe that for $y\in Y_1$, $\bfs_{\bgl_r,\iota(y)}=\bfs_{\bgl_{r-1},y}$. Thus, the notation $\bfs_y$ as this does not arise any confusion.

\begin{Lem}\label{lem:m to gl on w}
Let
\begin{itemize}
  \item $y=(y_1,y_2), y'=(y'_1,y'_2)$ with $y_1,y_1'\in Y_1$ and $y_2,y_2'\in Y_2$,
  \item $\ochi_1$ be an unramified character for $\overline{\GL_r}$ such that $\ochi_{1}|_{Y_{1,Q,n}^{sc}}=\ochi|_{Y_{Q,n}^{sc}}$.
\end{itemize}
If $y_2=y'_2$, then $\cw_{M,y}(y',\ochi)=\cw_{\GL_r,y_1}(y'_1,\ochi_1)$.
\end{Lem}

\begin{proof}
Recall that
\[
\cw_{M,y}(y',\ochi)=\sum_{\bw\in W(M)} c_M(\bw_M \bw,\ochi) \cdot \tau_M(\bw^{-1},{}^{\bw}\ochi,y,\bw_M(y')),
\]
It is straightforward to see that $c_M(\bw_M \bw,\ochi)=c_{\GL_r}(\bw_M \bw,\ochi_1)$.
By Lemma \ref{lem:10},
\[
\tau_M(\bw^{-1},{}^{\bw}\ochi,y,\bw_M(y')) =\tau_{\GL_r}(\bw^{-1},{}^{\bw}\ochi_1,y_1,\bw_{\GL_r}(y_1')).
\]
(Note that the condition $y_2=y_2'$ does not appear explicitly in the proof but must be satisfied.)
Therefore, $\cw_{M,y}(y',\ochi)=\cw_{\GL_r,y_1}(y'_1,\ochi_1)$.
\end{proof}

\subsection{Relative theta representations}

We now discuss the Whittaker models for the relative theta representations. In particular, we determine when these representations are non-generic and possess a unique Whittaker model. The main ingredient here is \cite{Gao17} Theorem 1.1, which is a generalization of \cite{KP84} Theorem I.3.5.

The root system spanned by $\Delta'$ is of type $A_{r_1-1}\times \cdots \times A_{r_k-1}$, where $r_1+\cdots+r_k=r$. In this way, we obtain a bijection between subsets of $\Delta$ and ordered partitions $(r_1\cdots r_k)$ of $r$. The following result can be proved along the same line as in \cite{Gao17}.

\begin{Thm}
\begin{enumerate}
\item If $r_i>n_Q$ for some $i$, then the representation $\Theta(\oM,\ochi)$ is non-generic.
\item If $r_i\leq n_Q$ for all $i$, then the representation $\Theta(\oM,\ochi)$ is generic.
\item If $r_i=n_Q$ for all $i$, then the representation $\Theta(\oM,\ochi)$ has a unique Whittaker model.
\end{enumerate}
\end{Thm}

\begin{proof}
The proof in \cite{Gao17} Example 3.16 and \cite{KP84} Corollary I.3.6 applies without essential change.
\end{proof}

By combining this result with Proposition \ref{prop:rodier}, we deduce the following result, in analogy with \cite{Suzuki98} Corollary 3.3.

\begin{Thm}\label{thm:whittaker for relative theta}
\begin{enumerate}
\item If $r_i>n_Q$ for some $i$, then the representation $\Theta(\oG/\oM,\ochi)$ is non-generic.
\item If $r_i\leq n_Q$ for all $i$, then the representation $\Theta(\oG/\oM,\ochi)$ is generic.
\item If $r_i=n_Q$ for all $i$, then the representation $\Theta(\oG/\oM,\ochi)$ has a unique Whittaker model.
\end{enumerate}
\end{Thm}

In the rest of this paper, we would like find a formula for some values of the unramified Whittaker functions in some special instances.

\subsection{Some calculation of $c(\bw,\ochi)$}

In this section, we carry out some calculation of $c(\bw,\ochi)$ for exceptional and anti-exceptional characters. In particular, we show that $\Theta(\oG/\oM,\ochi)$ contains a spherical vector.

\begin{Lem}\label{lem:gk for exceptiona}
Let $\ochi$ be an exceptional character. Then for $1\leq i\leq r-1$
\[
c(\sr_i\cdots \sr_1,\ochi)\neq 0.
\]
\end{Lem}

\begin{proof}
By direct calculation,
\[
c(\sr_i\cdots \sr_1,\ochi)=\dfrac{1-q^{-2}}{1-q^{-1}}\cdots \dfrac{1-q^{-(i+1)}}{1-q^{-i}}\neq 0.
\]
\end{proof}

\begin{Lem}
Let $\ochi$ be an exceptional character. Then
\[
c(\bw_G,\ochi)\neq 0.
\]
\end{Lem}

\begin{proof}
This follows the above lemma and induction.
\end{proof}

\begin{Cor}
The representation $\Theta(\oG/\oM,\ochi)$ contains a spherical vector.
\end{Cor}

\begin{proof}
The representation $\Theta(\oG/\oM,\ochi)$ is defined as the image of $T_{\bw_M,\ochi}:I(\ochi)\to I({}^{\bw_M}\ochi)$. The image contains the vector $c(\bw_G,\ochi)\phi_K^{{}^{\bw}\ochi}$, which is nonzero.
\end{proof}

\begin{Cor}\label{cor:vanishing of unramified whittaker}
If $r_i> n_Q$ for some $i$, then
\[
\cw_{\gamma}(\gamma',{}^{\bw_M}\ochi)=0
\]
for any $\gamma,\gamma'$.
\end{Cor}

\begin{proof}
If $\cw_{\gamma}(\gamma',{}^{\bw_M}\ochi)\neq 0$, then the Whittaker functional is nonzero on a spherical vector $\phi_K$.

We already know that $\phi_K\in\Theta(\oG/\oM,\ochi)$. Thus this implies that $\Theta(\oG/\oM,\ochi)$ has a nonzero Whittaker functional. This contradicts with our assumption and Theorem \ref{thm:whittaker for relative theta}.
\end{proof}

\begin{Lem}\label{lem:gk for anti 1}
Let $\ochi$ be an anti-exceptional character. Then
\[
c(\sr_i\cdots \sr_1,{}^{\sr_1\cdots \sr_i}\ochi)\neq 0.
\]
\end{Lem}

\begin{proof}
The proof is the same as Lemma \ref{lem:gk for exceptiona}. We have
\[
c(\sr_i\cdots \sr_1,{}^{\sr_1\cdots \sr_i}\ochi)=\dfrac{1-q^{-(i+1)}}{1-q^{-i}}\cdots \dfrac{1-q^{-2}}{1-q^{-1}}\neq 0.
\]
\end{proof}

\begin{Lem}\label{lem:gk for anti 2}
Let $\ochi$ be an anti-exceptional character. Let
\[
\bw=(\sr_{r-i+1}\cdots \sr_r)\cdots (\sr_2\cdots \sr_{i+1})(\sr_1\cdots \sr_i)
\]
Then
$c(\bw^{-1},{}^{\bw}\ochi)\neq 0$.
\end{Lem}

\begin{proof}
This follows from induction and Lemma \ref{lem:gk for anti 2}.
\end{proof}

\section{Calculation of certain local matrix coefficients}\label{sec:calculation of local matrix}

Assume that $n_Q \geq r$. The goal in this section is to calculation $\tau(\bw_G,{}^{\bw_G}\ochi,y,0)$ where $\ochi$ is an anti-exceptional character for $\overline{\GL}_{r}$. Recall that $\tau(\bw_G,{}^{\bw_G}\ochi,y,0)\neq 0$ unless $y=\bw[0]$ for some $\bw\in W$. As $n_Q\geq r$, the orbit of $0$ under the action of $W$ is free. The theta representation $\Theta(\oG,{}^{\bw_G}\ochi)$ is realized as a subrepresentation of $I(\ochi)$. Recall that $\phi_K\in \Theta(\oG,{}^{\bw_G}\ochi)$.

\begin{Lem}\label{lem:3}
For $\bw \in W$,
\[
\cw_{\bw[0]}(0,\ochi)= \tau(\bw_G,{}^{\bw_G}\ochi,\bw[0],0).
\]
\end{Lem}
\begin{proof}
This follows from Theorem \ref{thm:unramified whittaker function}. Note that $c(\bw_G\bw,\ochi)=0$ unless $\bw=\bw_G$.
\end{proof}

Let $y\in Y$. 
We define the Gauss sum $g(\bw,y)$ for $\bw\in W$ as follows:
\begin{enumerate}
\item $g(\id,y)=1$;
\item For a simple reflection $\sr_\al$,
\[
g(\sr_\al,y)=\bfg_{\psi^{-1}}(\la y_\rho,\al\ra Q(\al^\vee)).
\]
\item If $\bw_1,\bw_2\in W$ such that $\ell(\bw_1\bw_2)=\ell(\bw_1)+\ell(\bw_2)$, then
\[
g(\bw_1\bw_2,y)=g(\bw_1,\bw_2[y])g(\bw_2,y).
\]
\end{enumerate}

We have to verify that this is well-defined.
\begin{Lem}
We have
\[
g(\bw,y)=\prod_{\al \in \Phi(\bw)} \bfgg_{\psi^{-1}}(\la y_\rho,\al\ra Q(\al^\vee)).
\]
Therefore, $g(\bw,y)$ is well-defined.
\end{Lem}

\begin{proof}
Recall that $\bw[y]_\rho=\bw(y-\rho)$ for any $\bw \in W$.
Fix a reduced decomposition $\bw=\sr_{i_1}\cdots\sr_{i_k}$. Then
\[
\begin{aligned}
g(\bw,y)=&g(\sr_{i_1},\sr_{i_2}\cdots \sr_{i_k}[y])\cdots g(\sr_{i_{k-1}},\sr_{i_k}[y])g(\sr_{i_k},y)\\
=&\prod_{j=1}^k \bfgg_{\psi^{-1}}(\la \sr_{i_{j+1}}\cdots\sr_{i_k}(y-\rho),\al_{i_j}\ra Q(\al^\vee))\\
=&\prod_{j=1}^k \bfgg_{\psi^{-1}}(\la y-\rho,\sr_{i_{k}}\cdots\sr_{i_{j+1}}(\al_{i_j})\ra Q(\al^\vee))\\
=&\prod_{\al \in \Phi(\bw)} \bfgg_{\psi^{-1}}(\la y-\rho,\al\ra Q(\al^\vee)).
\end{aligned}
\]
The last equality follows from \cite{Bump13} Proposition 20.10.
\end{proof}

We can now state the main result of this section.

\begin{Prop}\label{prop:tau in the stable range}
For $\bw \in W$,
\[
\tau(\bw_G,{}^{\bw_G}\ochi,\bw[0],0)=g(\bw,0).
\]
\end{Prop}

The rest of this section is devoted to proving this result. Before the proof, we need some preparation.

\subsection{Two lemmas}
\begin{Lem}\label{lem:1}
We have $\tau(\sr_\al,\ochi,y,\sr_\al[y])=\bfg_{\psi^{-1}}(-\la y_\rho,\al\ra Q(\al^\vee))$.
\end{Lem}
\begin{proof}
This is done by direct calculation. Recall $\bw[y]=\bw(y-\rho)+\rho$. The left-hand side is
\[
\begin{aligned}
&\tau(\sr_\al,\ochi,\sr_\al[\sr_\al[y]],\sr_\al[y])\\
=&\bfg_{\psi^{-1}}(\la \sr_\al[y]-\rho,\al\ra Q(\al^\vee))
=\bfg_{\psi^{-1}}(\la y-\rho,\sr_\al(\al)\ra Q(\al^\vee))\\
=&\bfg_{\psi^{-1}}(\la y-\rho,-\al \ra Q(\al^\vee))
=\bfg_{\psi^{-1}}(-\la y_\rho,\al\ra Q(\al^\vee)).
\end{aligned}
\]
\end{proof}

\begin{Lem}\label{lem:5}
If $\ell(\sr_\al \bw)=\ell(\bw)+1$, then $\la \bw[0],\al\ra \leq 0$; if $\ell(\sr_\al \bw)=\ell(\bw)-1$, then $\la \bw[0],\al\ra>0$.
\end{Lem}

\begin{proof}
If $\ell(\sr_\al \bw)=\ell(\bw)+1$, then $\bw^{-1}(\al)$ is a positive root (\cite{Bump13} Proposition 20.2). This implies that
$\la \bw(\rho),\al\ra=\la \rho,\bw^{-1}(\al)\ra \geq 1$. Note that $\la \rho,\al^\vee\ra=1$. Thus
\[
\la \rho-\bw(\rho),\al\ra=\la \rho,\al\ra-\la \bw(\rho),\al\ra\leq 0.
\]

We now consider the other case. Note
\[
(\sr_\al \bw)[0]=\sr_\al \bw(-\rho)+\rho=\sr_\al(\bw(-\rho)+\rho)-\sr_\al(\rho)+\rho =\sr_\al(\bw[0])-\sr_\al[0].
\]
If $\ell(\sr_\al \bw)=\ell(\bw)-1$ or $\ell(\sr_\al \bw)+1=\ell(\bw)$, then $\la (\sr_\al \bw)[0],\al\ra\leq 0$. But
\[
\la (\sr_\al \bw)[0],\al\ra=\la \sr_\al(\bw[0])-\sr_\al[0],\al\ra=-\la \bw[0],\al\ra+2.
\]
Here we use $\la \sr_\al[0],\al\ra =-2$.  This gives the desired result.
\end{proof}

\subsection{Proof of Proposition \ref{prop:tau in the stable range}}
We first check some small rank cases. If $r=1$, then both sides are $1$. If $r=2$, we only have two Weyl group elements two consider. If $\bw=\id$, then
\[
\tau(\bw_G,{}^{\bw_G}\ochi,\bw[0],0)=\dfrac{1-q^{-1}}{1-q^{-1}}=1;
\]
if $\bw=\bw_G$, then clearly
\[
\tau(\bw_G,{}^{\bw_G}\ochi,\bw_G[0],0)=g(\bw_G,0).
\]

Our proof is a simplified version of the proof of \cite{Suzuki97} Lemma 4.2.
We now assume that the result is true for $r$ and prove it for $r+1$. We first apply the inductive formula in Proposition \ref{prop:inductive 1 in general linear}.  Observe the following:

\begin{itemize}
\item Recall that $\bw^{-1}[\bw[0]]=0$ and we write $\bw^{-1}=\bw^{'-1}\sr_r\cdots \sr_{r_0}$ for a unique integer $r_0$ and $\bw'\in W(\bm)$.
  \item We are working with the exceptional representation. If $i\neq 1$, then $\prod_{j=1}^{i-1}\frac{1-q^{-1}\ochi_{ji}}{1-\ochi_{ji}}=0$ since $1-q^{-1}\ochi_{i-1,i}=0$. Thus only one term $(i=1)$ in the outer summation is nonzero.
\end{itemize}
We obtain
\[
\cw_{y}(0,\ochi)=\sum_{y'}
\tau(\sr_1\cdots \sr_r,{}^{\sr_r\cdots \sr_1}\ochi,y,y')
\cw_{M,y'}(0,{}^{\sr_r\cdots \sr_1}\ochi),
\]
where the sum is over the set
\[
\begin{aligned}
&\left\{
\sr_r\cdots \sr_{r_0} \sr_{r_0-2}^{a_{r_0-2}}\cdots \sr_i^{a_1}[y]: a_1,\cdots,a_{r_0-2}\in \{0,1\}
\right\}\\
=&\left\{
\sr_{r_0-2}^{a_{r_0-2}}\cdots \sr_i^{a_1}\bw'[0]: a_1,\cdots,a_{r_0-2}\in \{0,1\}
\right\}.\\
\end{aligned}
\]
Since the orbit of $0$ (hence $y$) is free, this set has no repetition for different $(a_1,\cdots,a_{r_0-2})$.
Thus we now obtain
\[
\cw_{y}(0,\ochi)=\sum_{a_1\in \{0,1\}}\cdots \sum_{a_{r_0-2}\in \{0,1\}} \tau(\sr_1\cdots \sr_r,{}^{\sr_r\cdots \sr_1}\ochi,  y,y')
\cw_{M,y'}(0,{}^{\sr_r\cdots \sr_1}\ochi)
\]
where
\[
y'=\sr_{r}\cdots \sr_{r_0}\sr_{r_0-2}^{a_{r_0-2}}\cdots \sr_1^{a_1} \bw[0]=\sr_{r_0-2}^{a_{r_0-2}}\cdots \sr_1^{a_1} \bw'[0].
\]

For $y'=\sr_{r_0-2}^{a_{r_0-2}}\cdots \sr_1^{a_1} \bw'[0]$, using Lemma \ref{lem:tau unique decomposition}, we see that $\tau(\sr_1\cdots \sr_r,{}^{\sr_r\cdots \sr_1}\ochi,  y,y')$ is the product of the following three terms:
\begin{itemize}
  \item $\tau(\sr_1\cdots \sr_{r_0-2},~ {}^{\sr_{r_0-2}\cdots \sr_1}\ochi,~ y,~\sr_{r_0-2}^{a_{r_0-2}}\cdots \sr_1^{a_1} \bw[0])$,
  \item $\tau(\sr_{r_0-1},~ {}^{\sr_{r_0-1}\cdots \sr_1}\ochi, ~\sr_{r_0-2}^{a_{r_0-2}}\cdots \sr_1^{a_1} \bw[0], ~\sr_{r_0-2}^{a_{r_0-2}}\cdots \sr_1^{a_1} \bw[0])$,
  \item $\tau(\sr_{r_0}\cdots \sr_{r}, ~ {}^{\sr_r\cdots \sr_1}\ochi,~ \sr_{r_0-2}^{a_{r_0-2}}\cdots \sr_1^{a_1} \bw[0], ~ y')$.
\end{itemize}

We now analyze each term. We start with $\cw_{M,y'}(0,{}^{\sr_r\cdots \sr_1}\ochi)$.

\begin{Lem}
We have
\[
\begin{aligned}
&\cw_{M,y'}(0,{}^{\sr_r\cdots \sr_1}\ochi)\\
=&g(\sr_{r_0-2}^{a_{r_0-2}}\cdots \sr_1^{a_1} \bw',0).\\
\end{aligned}
\]
\end{Lem}

\begin{proof}
Here we apply Lemma \ref{lem:m to gl on w}. We observe that the character ${}^{\sr_r\cdots \sr_1}\ochi$ restricted to $\oT_{\GL_r,Q,n}^{sc}$ is again an anti-exceptional character. So we can apply induction to calculate the value. It is $g(\sr_{r_0-2}^{a_{r_0-2}}\cdots \sr_1^{a_1} \bw',0)$.
\end{proof}

\begin{Lem}
We have
\[
\tau(\sr_{r_0-1},{}^{\sr_{r_0-1}\cdots \sr_1}\ochi,\sr_{r_0-2}^{a_{r_0-2}}\cdots \sr_1^{a_1} \bw[0],\sr_{r_0-2}^{a_{r_0-2}}\cdots \sr_1^{a_1} \bw[0])
=\dfrac{1-q^{-1}}{1-q^{-(r_0-1)}}.
\]
\end{Lem}
\begin{proof}
Note $\sr_{r_0-2}^{a_{r_0-2}}\cdots \sr_1^{a_1} \bw$ is of the form $\sr_{r_0}\cdots \sr_{r}\bw''$ with $\bw''\in W(\bgl_r)$. 
Lemma \ref{lem:5} shows that $\la \sr_{r_0}\cdots \sr_{r-1}\bw''[0],\al_{r_0-1}\ra<0$. Therefore $k_{\sr_{r_0}\cdots \sr_{r-1}\bw''[0],\al_{r_0-1}}=0$ and
\[
\tau(\sr_{r_0-1},{}^{\sr_{r_0-1}\cdots \sr_1}\ochi,\sr_{r_0-2}^{a_{r_0-2}}\cdots \sr_1^{a_1} \bw[0],\sr_{r_0-2}^{a_{r_0-2}}\cdots \sr_1^{a_1} \bw[0])
=\dfrac{1-q^{-1}}{1-q^{-(r_0-1)}}.
\]
\end{proof}

\begin{Lem}
We have
\[
\tau(\sr_{r_0}\cdots \sr_{r},{}^{\sr_r\cdots \sr_1}\ochi,\sr_{r_0-2}^{a_{r_0-2}}\cdots \sr_1^{a_1} \bw[0],y')=g(\sr_{r_0}\cdots \sr_{r},\bw'[0]).
\]
\end{Lem}

\begin{proof}
Since the action of $\sr_{r_0}\cdots \sr_{r}$ and $\sr_{r_0-2}^{a_{r_0-2}}\cdots \sr_1^{a_1}$ are disjoint, we have
\[
\begin{aligned}
&\tau(\sr_{r_0}\cdots \sr_{r},{}^{\sr_r\cdots \sr_1}\ochi,\sr_{r_0-2}^{a_{r_0-2}}\cdots \sr_1^{a_1} \bw[0],y')\\
=&\tau(\sr_{r_0}\cdots \sr_{r},{}^{\sr_r\cdots \sr_1}\ochi,\sr_{r_0}\cdots \sr_{r}\bw'[0],\bw'[0]).
\end{aligned}
\]
We can now use Lemma \ref{lem:tau unique decomposition} to calculate
\[
\begin{aligned}
&\tau(\sr_{r_0}\cdots \sr_{r},{}^{\sr_r\cdots \sr_1}\ochi,\sr_{r_0}\cdots \sr_{r}\bw'[0],\bw'[0])\\
=&\tau(\sr_{r_0},{}^{\sr_{r_0}\cdots \sr_1}\ochi,\sr_{r_0}\cdots \sr_{r}\bw'[0],\sr_{r_0-1}\cdots \sr_{r}\bw'[0])
\cdots \tau(\sr_{r},{}^{\sr_r\cdots \sr_1}\ochi,\sr_{r}\bw'[0],\bw'[0])\\
=&g(\sr_{r_0},\sr_{r_0-1}\cdots \sr_{r}\bw'[0])\cdots g(\sr_{r},\bw'[0])\\
=&g(\sr_{r_0}\cdots \sr_{r},\bw'[0]).
\end{aligned}
\]
\end{proof}

Let us summarize what we have done so far. Let us rewrite
\[
g(\sr_{r_0-2}^{a_{r_0-2}}\cdots \sr_1^{a_1} \bw',0)
= \dfrac{g(\sr_{r_0-2}^{a_{r_0-2}}\cdots \sr_1^{a_1} \bw',0)}{g(\bw',0)}g(\bw',0).
\]
and use
\[
g(\bw,0)=g(\sr_{r_0}\cdots \sr_{r},\bw'[0])g(\bw',0).
\]
By the above results, we deduce that
\begin{equation}\label{eq:rewrite as a sum}
\begin{aligned}
&\cw_{y}(0,\ochi)=g(\bw,0)\dfrac{1-q^{-1}}{1-q^{-(r_0-1)}}\\
\cdot&\sum_{a_1\in \{0,1\}}\cdots \sum_{a_{r_0-2}\in \{0,1\}} \tau (\sr_1\cdots \sr_{r_0-2},{}^{\sr_{r_0-2}\cdots \sr_1}\ochi,y,\sr_{r_0-2}^{a_{r_0-2}}\cdots \sr_1^{a_1} \bw[0])
\dfrac{g(\sr_{r_0-2}^{a_{r_0-2}}\cdots \sr_1^{a_1} \bw',0)}{g(\bw',0)}.
\end{aligned}
\end{equation}

Thus, it remains to the summation in the second line. Notice that
\[
\dfrac{g(\sr_{r_0-2}^{a_{r_0-2}}\cdots \sr_1^{a_1} \bw',0)}{g(\bw',0)}=\dfrac{g(\sr_{r_0-2}^{a_{r_0-2}}\cdots \sr_1^{a_1} \bw,0)}{g(\bw,0)}.
\]
We now rewrite the summation as
\begin{equation}\label{eq:break the sum}
\begin{aligned}
&\sum_{a_1\in \{0,1\}}\cdots \sum_{a_{r_0-3}\in \{0,1\}}
\tau(\sr_1\cdots \sr_{r_0-3},{}^{\sr_{r_0-3}\cdots \sr_1}\ochi,y,\sr_{r_0-3}^{i_{r_0-3}}\cdots \sr_1^{a_1} \bw[0]) \dfrac{g(\sr_{r_0-3}^{a_{r_0-3}}\cdots \sr_1^{a_1} \bw,0)}{g(\bw,0)} \\
\cdot &\sum_{a_{r_0-2}\in \{0,1\}}
\tau(\sr_{r_0-2},{}^{\sr_{r_0-2}\cdots \sr_1}\ochi, \sr_{r_0-3}^{i_{r_0-3}}\cdots \sr_1^{a_1} \bw[0],\sr_{r_0-2}^{a_{r_0-2}}\cdots \sr_1^{a_1} \bw[0])
\dfrac{g(\sr_{r_0-2}^{a_{r_0-2}}\cdots \sr_1^{a_1} \bw,0)}{g(\sr_{r_0-3}^{a_{r_0-3}}\cdots \sr_1^{a_1} \bw,0)}.
\end{aligned}
\end{equation}
We first calculate the inner sum.

\begin{Lem}
The inner sum in \eqref{eq:break the sum} is equal to $\dfrac{1-q^{-(r_0-1)}}{1-q^{-(r_0-2)}}$.
\end{Lem}

\begin{proof}
To calculate the inner sum, there are two cases to consider. (Note that this discussion does not appear in \cite{Suzuki97} Sect. 4.2.)

For ease of notation, we write $\tilde y=\sr_{r_0-3}^{i_{r_0-3}}\cdots \sr_1^{a_1} \bw[0]$.
Clearly, $k_{\tilde y,\al_{r_0-2}}$ is either $0$ or $1$. We have two cases to consider.

\textit{Case 1}: $\ell(\sr_{r_0-2}\sr_{r_0-3}^{i_{r_0-3}}\cdots \sr_1^{a_1} \bw)=\ell(\sr_{r_0-3}^{i_{r_0-3}}\cdots \sr_1^{a_1} \bw)+1$.

When $a_{r_0-2}=0$,
\[
\dfrac{g(\sr_{r_0-2}^{a_{r_0-2}}\cdots \sr_1^{a_1} \bw,0)}{g(\sr_{r_0-3}^{a_{r_0-3}}\cdots \sr_1^{a_1} \bw,0)}=1.
\]
By Lemma \ref{lem:5}, $\la \tilde y,\al_{r_0-2}\ra\leq 0$ and $k_{\tilde y,\al_{r_0-2}}=0$. Thus,
\[
\tau(\sr_{r_0-2},{}^{\sr_{r_0-2}\cdots \sr_1}\ochi,\sr_{r_0-3}^{a_{r_0-3}}\cdots \sr_1^{a_1} \bw[0],\sr_{r_0-2}^{a_{r_0-2}}\cdots \sr_1^{a_1} \bw[0])
=\dfrac{1-q^{-1}}{1-q^{-(r_0-2)}}.
\]
When $a_{r_0-2}=1$,
\[
\dfrac{g(\sr_{r_0-2}^{a_{r_0-2}}\cdots \sr_1^{a_1} \bw,0)}{g(\sr_{r_0-3}^{a_{r_0-3}}\cdots \sr_1^{a_1} \bw,0)}=g(\sr_{r_0-2},\sr_{r_0-3}^{i_{r_0-3}}\cdots \sr_1^{a_1} \bw[0])=\bfgg_{\psi^{-1}}(\la \tilde y_\rho,\al_{r_0-2}\ra Q(\al^\vee)),
\]
and
\[
\tau(\sr_{r_0-2},{}^{\sr_{r_0-2}\cdots \sr_1}\ochi,\sr_{r_0-3}^{i_{r_0-3}}\cdots \sr_1^{a_1} \bw[0],\sr_{r_0-2}^{a_{r_0-2}}\cdots \sr_1^{a_1} \bw[0]
)=\bfgg_{\psi^{-1}}(-\la \tilde y_\rho,\al_{r_0-2}\ra Q(\al^\vee)).
\]
Thus  the inner sum in \eqref{eq:break the sum} is
\[
\begin{aligned}
&\dfrac{1-q^{-1}}{1-q^{-(r_0-2)}}\cdot 1 +\bfgg_{\psi^{-1}}(\la \tilde y_\rho,\al_{r_0-2}\ra Q(\al^\vee))\cdot \bfgg_{\psi^{-1}}(-\la \tilde y_\rho,\al_{r_0-2}\ra Q(\al^\vee))\\
=&\dfrac{1-q^{-1}}{1-q^{-(r_0-2)}}+q^{-1}=\dfrac{1-q^{-(r_0-1)}}{1-q^{-(r_0-2)}}.
\end{aligned}
\]

\textit{Case 2}: $\ell(\sr_{r_0-2}\sr_{r_0-3}^{i_{r_0-3}}\cdots \sr_1^{a_1} \bw)=\ell(\sr_{r_0-3}^{i_{r_0-3}}\cdots \sr_1^{a_1} \bw)-1$. When $a_{r_0-2}=0$,
\[
\dfrac{g(\sr_{r_0-2}^{a_{r_0-2}}\cdots \sr_1^{a_1} \bw,0)}{g(\sr_{r_0-3}^{a_{r_0-3}}\cdots \sr_1^{a_1} \bw,0)}=1.
\]
By Lemma \ref{lem:5}, $\la \tilde y,\al_{r_0-2}\ra>0$ and $k_{\tilde y,\al_{r_0-2}}=1$. Thus,
\[
\tau(\sr_{r_0-2},{}^{\sr_{r_0-2}\cdots \sr_1}\ochi,\sr_{r_0-3}^{a_{r_0-3}}\cdots \sr_1^{a_1} \bw[0],\sr_{r_0-2}^{a_{r_0-2}}\cdots \sr_1^{a_1} \bw[0])
=\dfrac{1-q^{-1}}{1-q^{-(r_0-2)}}q^{-(r_0-2)}.
\]
When $a_{r_0-2}=1$,
\[
\tau(\sr_{r_0-2},{}^{\sr_{r_0-2}\cdots \sr_1}\ochi,\sr_{r_0-3}^{i_{r_0-3}}\cdots \sr_1^{a_1} \bw[0],\sr_{r_0-2}^{a_{r_0-2}}\cdots \sr_1^{a_1} \bw[0])=\bfgg_{\psi^{-1}}(\la \tilde y_\rho,\al_{r_0-2}\ra Q(\al^\vee)),
\]
and
\[
\dfrac{g(\sr_{r_0-2}^{a_{r_0-2}}\cdots \sr_1^{a_1} \bw,0)}{g(\sr_{r_0-3}^{a_{r_0-3}}\cdots \sr_1^{a_1} \bw,0)}=\dfrac{1}{g(\sr_{r_0-2},\sr_{r_0-2}^{a_{r_0-2}}\cdots \sr_1^{a_1} \bw[0])}=\dfrac{1}{\bfgg_{\psi^{-1}}(\la \tilde y_\rho,\al_{r_0-2}\ra Q(\al^\vee))}.
\]
Then  the inner sum in \eqref{eq:break the sum} is
\[
\begin{aligned}
&\dfrac{1-q^{-1}}{1-q^{-(r_0-2)}}q^{-(r_0-2)}+\bfgg_{\psi^{-1}}(\la \tilde y_\rho,\al_{r_0-2}\ra Q(\al^\vee))\cdot \dfrac{1}{\bfgg_{\psi^{-1}}(\la \tilde y_\rho,\al_{r_0-2}\ra Q(\al^\vee))}\\
=&\dfrac{1-q^{-1}}{1-q^{-(r_0-2)}}q^{-(r_0-2)}+1=\dfrac{1-q^{-(r_0-1)}}{1-q^{-(r_0-2)}}.
\end{aligned}
\]
\end{proof}

This finishes the calculation of the inner sum in \eqref{eq:break the sum}.
We can now proceed for the other summations and deduce that \eqref{eq:break the sum} is
\[
\dfrac{1-q^{-2}}{1-q^{-1}}\cdots \dfrac{1-q^{-(r_0-2)}}{1-q^{-(r_0-3)}}\dfrac{1-q^{-(r_0-1)}}{1-q^{-(r_0-2)}} =\dfrac{1-q^{-(r_0-1)}}{1-q^{-1}}.
\]
By \eqref{eq:rewrite as a sum}, this implies that $\cw_{y}(0,\ochi)=g(\bw,0)$. The proof of Proposition \ref{prop:tau in the stable range} is complete.

\section{Main result}\label{sec:main result}

\subsection{Statement}

We now state our main result. We work with the group $\overline{\GL}_{r}$.
Let $\Delta'\subset \Delta$ so that the corresponding Levi subgroup is $\bm=\bgl_{r_1}\times \cdots \times \bgl_{r_k}$. Let $\ochi$ be an $\Delta'$-anti-exceptional character for $\overline{\GL}_r$. Define $e_i=r_1+\cdots+r_{i-1}$ and  $x_{ij}=\ochi_{e_i+1,e_j+1}$.

\begin{Thm}\label{thm:final formula}
Assume that
\begin{itemize}
\item for all $1\leq i\leq k$, $r_i\leq n_Q$,
\item for all $1\leq i,j\leq k$, $r_i+r_j>n_Q$
\item for $i=1,\cdots,k-1$,
\begin{equation}\label{eq:strange condition}
\left\lfloor \dfrac{(\sum_{j=1}^i r_j)(\sum_{j=i+1}^k r_j)}{n_Q}\right\rfloor=
(k-i)\left(\sum_{j=1}^i r_j\right)+
i\left(\sum_{j=i+1}^k r_j\right)
-i(k-i)n_Q.
\end{equation}
\end{itemize}
Then
\[
\cw_{\bw[0]}(0,\ochi)=g(\bw,0)\prod_{1\leq i<j\leq k}\ \prod_{l=n_Q-r_j+1}^{r_i}(1-x_{ij}q^{-l})
\]
for $\bw\in W(\bm)$.
\end{Thm}

We will first prove the result for $\bw=\id$ and then for the general case. We begin with some remarks in the case of $\bw=\id$.
\begin{Rem}
\begin{enumerate}
  \item A result of \cite{McNamara11} says that $\cw_0(0,\ochi)$ is a weighted sum over a finite crystal graph and is therefore a polynomial in $x_{12},\cdots,x_{k-1,k}$. Note that everything stated here is done in $\overline{\SL}_r$ so McNamara's result does apply.
  \item When $\bw=\id$, we can rewrite the right-hand side as a polynomial in $x_{12},\cdots, x_{k-1,k}$. Let $f(x_{12},\cdots,x_{k-1,k})$ be this polynomial. The monomial with highest total degree is
\[
\prod_{1\leq i<j\leq k}x_{ij}^{r_i+r_j-n}=x_{12}^{b_1}\cdots x_{k-1,k}^{b_{k-1}}
\]
where
\[
b_i=(k-i)\left(\sum_{j=1}^i r_j\right)+i\left(\sum_{j=i+1}^k r_j\right)-i(k-i)n_Q
\]
is the right-hand side of \eqref{eq:strange condition}.
\item The condition in \eqref{eq:strange condition} does seem strange and this is not satisfied for all tuples. However, it is easy to check that \eqref{eq:strange condition} holds when $(r_1,\cdots,r_k)=(n_Q,\cdots,n_Q,n')$ where $1\leq n'\leq n_Q$.
\item We expect the result to be true without the condition in \eqref{eq:strange condition}. But we do not know how to extend it at the moment.
\end{enumerate}
\end{Rem}

\subsection{Proof of Theorem \ref{thm:final formula}: the base case}

For the base case, the proof presented here is adapted from \cite{Kaplan} Theorem 43. We will give also examples to explain some ideas and give the reader some flavor of the proof.

We now give an outline of the proof. We first observe that, by the results in \cite{McNamara11}, $\cw_0(0,\ochi)$ is weight sum over certain Gelfand-Tsetlin patterns and is therefore a polynomial in $x_{12},\cdots,x_{k-1,k}$. 

It is sufficient to prove the following three things:
\begin{enumerate}
\item Every factor of $f(x_{12},\cdots,x_{k-1,k})$ divides $\cw_0(0,\ochi)$.
\item The monomial of the highest total degree of $\cw_{0}(0,\ochi)$ is the same as $f(x_{12},\cdots,x_{k-1,k})$, up to a scalar.
\item The constant coefficient of $f(x_{12},\cdots,x_{k-1,k})$ is $1$. So it is enough to prove that the constant coefficient of $\cw_{0}(0,\ochi)$ is $1$.
\end{enumerate}

The first one is proved by a representation-theoretic argument. The last two are proved using the formula of \cite{McNamara11} Sect. 8, which is based on the Gelfand-Tsetlin description of \cite{BBF11} Sect. 8. Note that the proof in \cite{Kaplan} Theorem 43 does not use uniqueness of Whittaker models.

We start with the representation-theoretic argument.

\begin{Ex}\label{ex:main section}
We assume that $n=3$, $r=8$, $(r_1,r_2,r_3)=(3,3,2)$, $Q(\al^\vee)=1$. Let $\Delta'=\{\al_1,\al_2,\al_4,\al_5,\al_7\}$ and $\ochi$ be a $\Delta'$-anti-exceptional character. Therefore, $\ochi_{\al}=q$ for $\al\in \Delta'$. Let $x_1=\ochi_{\al_1+\al_2+\al_3}$ and $x_2=\ochi_{\al_4+\al_5+\al_6}$. It is easy to check, for instance, $\ochi_{\al_3}=q^{-2}x_1$.

Clearly if $x_1=q^3$, then $\ochi$ is a $\Delta'\cup \{\al_3\}$-anti-exceptional character and $\cw_0(0,\ochi)=0$ by Corollary \ref{cor:vanishing of unramified whittaker}. In other words, as a function of $x_1,x_2$, $\cw_0(0,\ochi)$ is zero along the hyperplane $1-q^{-3}x_1=0$.

Now let us consider the following question: is $\cw_0(0,\ochi)$ along other hyperplanes? A quick examination shows that $1-q^{-2}x_1=0$ does the job. In fact, under this assumption, $\ochi_{\al_3+\al_4}=q$. Thus ${}^{\sr_4}\ochi$ is an $\{\al_1,\al_2,\al_3\}$-anti-exceptional character. We consider the following intertwining operator $T_{\sr_4,{}^{\sr_4}\ochi}:I({}^{\sr_4}\ochi)\to I(\ochi)$.
Using Lemma \ref{lem:gk for anti 2}, it is easy to check
\[
T_{\sr_4,{}^{\sr_4}\ochi}(\phi_K)=\dfrac{1-q^{-1}}{1-q^{-2}}\phi'_K.
\]
If $\cw_0(0,\ochi)\neq 0$, then by composing this Whittaker functional with $T_{\sr_4,{}^{\sr_4}\ochi}$, we obtain a nonzero Whittaker functional on $I({}^{\sr_4}\ochi)$. However, this contradicts with Corollary \ref{cor:vanishing of unramified whittaker} as ${}^{\sr_4}\ochi$ is an $\{\al_1,\al_2,\al_3\}$-anti-exceptional character.

The same argument shows that $\cw_0(0,\ochi)=0$ if $1-q^{-1}x_1=0$. The same argument can be applied for $x_2$.

We now consider the hyperplane $1-\ochi_{\al_1+\cdots+\al_6}q^{-2}=1-x_1x_2q^{-2}=0$. With this assumption, $\ochi_{\al_3+\cdots+\al_7}=q$. Therefore, ${}^{\sr_4\cdots\sr_7}\ochi$ is $\{\al_1,\al_2,\al_3\}$-anti-exceptional. The intertwining operator
\begin{equation}\label{eq:intertwining in examples}
T_{\sr_7\cdots\sr_4,{}^{\sr_4\cdots\sr_7}\ochi}:I({}^{\sr_4\cdots\sr_7}\ochi)\to I(\ochi)
\end{equation}
could have zeros. But $c(\sr_7\cdots\sr_4,{}^{\sr_4\cdots\sr_7}\ochi)$ has two types of factors: the first of the form $1-x_2q^{-l}$ for some integer $l$ and the factor as in the statement of Lemma \ref{lem:gk for anti 2}. In any case, \eqref{eq:intertwining in examples} is nonzero on spherical vectors along $1-x_1x_2q^{-2}=0$. Now the same argument as above shows that $\cw_0(0,\ochi)=0$.

By repeating this argument, one can find $3+2+2$ factors of $\cw_0(0,\ochi)$. They are exactly the factors appearing in the statement of Theorem \ref{thm:final formula}.
\end{Ex}

\begin{Lem}
If $1-x_{ij}q^{-l}=0$ for $i< j$ and $n_Q-r_j+1\leq l\leq r_i$, then $\cw_{0}(0,\ochi)=0$.
\end{Lem}

\begin{proof}
We write $\ochi\sim(\ochi_1,\cdots, \ochi_k)$ where $\ochi_m$ is anti-exceptional for the group $\overline{\GL}_{r_m}$. We further write $\ochi_j\sim (\ochi_j^\dagger, \ochi_j^\ddagger)$  where the size of $\ochi_j^\dagger$ is $r_i-l$ (which could be $0$). Let $\bw$ be the Weyl group element so that
\[
{}^{\bw}\ochi\sim (\ochi_1, \cdots, \ochi_i, \ochi_j^\ddagger,\cdots, \ochi_{j-1},  \ochi_j^\dagger, \ochi_{j+1}, \cdots, \ochi_k).
\]
Observe that since $1-x_{ij}q^{-1}=0$, $(\ochi_i,\ochi_j^\ddagger)$ is an anti-exceptional character of size $r_i+r_j-r_i+l=r_j+l\geq n_Q+1$. Thus ${}^{\bw}\ochi$ is an anti-exceptional character that satisfies the condition in Corollary \ref{cor:vanishing of unramified whittaker}.

We now check that the intertwining operator
\[
T_{\bw^{-1},{}^{\bw}\ochi}:I({}^{\bw}\ochi)\to I(\ochi)
\]
is nonzero on spherical vectors along $1-x_{ij}q^{-l}=0$. Indeed, it is enough check that $c(\bw^{-1},{}^{\bw}\ochi)\neq 0$ along $1-x_{ij}q^{-l}=0$. A quick calculation shows that the denominator of $c(\bw^{-1},{}^{\bw}\ochi)$ is either of the form $1-x_{jj'}q^{-l'}$ for $j'\neq i,j$, or a factor of the form as in Lemma \ref{lem:gk for anti 2}. In either case, this is nonzero when $1-x_{ij}q^{-l}=0$.

Suppose now that $\cw_0(0,\ochi)\neq 0$ along $1-x_{ij}q^{-l}=0$. We then have a Whittaker functional on $I({}^{\bw}\ochi)$ via
\[
I({}^{\bw}\ochi)\to I(\ochi)\to \bc.
\]
This is nonzero since it is nonzero on the spherical vector $\phi_K$. However, by our discussion above, this contradicts Corollary \ref{cor:vanishing of unramified whittaker}.
\end{proof}

Thus we know that $(1-\chi_{ij}q^{-l})\mid \cw_0(0,\ochi)$. As the factors of $f(x_{12},\cdots,x_{k-1,k})$ are distinct and $\bc[x_{12},\cdots,x_{k-1,k}]$ is a unique factorization domain, $f(x_{12},\cdots,x_{k-1,k})$ divides $\cw_{0}(0,\ochi)$.

We now use the formula in \cite{McNamara11} to estimate the degree of $\cw_0(0,\ochi)$. We briefly recall how this is derived. In \cite{McNamara11}, with a choice of a reduced decomposition of the longest element in the Weyl group, McNamara introduces an algorithm, called \textit{explicit Iwasawa decomposition}. This is to write an element $u\in U^-$ as $u=tnk$ where $t\in T$, $n\in U$ and $k\in K$. Equivalently, this is to write $U^-$ as a cell decomposition $U^{-}=\bigsqcup_{\mathbf{m}}C_{\mathbf{m}}$, where $\mathbf{m}$ is a tuple of integers. Thus one can write
\[
\int_{U^-}=\sum_{\mathbf{m}}\int_{C_{\mathbf{m}}}
\]
and this yields a combinatorial sum of these integrals. The main result in \cite{McNamara11} says that unramified Whittaker functions can be calculated in this way, and the tuples $\mathbf{m}$ with nonzero contributions are in bijection with a set of Gelfand-Tsetlin patterns. The contribution can be calculated in terms of Gauss sums.

Recall the a strict Gelfand-Tsetlin pattern is a triangular array $\{a_{i,j}\}$ of non-negative integers, such that each row is strictly decreasing and $a_{i,j} \geq a_{i+1,j}\geq a_{i_{j+1}}$ for all $i,j$ such that all entries exist. For each $i$, define
\[
d_i:=\sum_{j=1}^{r-i+1}a_{i,j}-a_{1,i+j-1}=\sum_{j=1}^{r-i+1}a_{i,j}-(r-i-j+1).
\]

Here $\cw_{0}(0,\ochi)$ is expressed as a sum over the set of Gelfand-Tsetlin patterns with the first row $a_{1,j}=r-j, 1\leq j\leq r$. The resulting monomial for such a pattern is of the form
\[
C\prod_{l=1}^{k-1}x_{l,l+1}^{d_{e_{l+1}+1}/n_Q},
\]
where $C$ is a certain product of powers of $q$ and Gauss sums.

In this paper, we choose a particular maximal abelian subgroup $\oA$. This imposes another condition on the patterns we need to consider. With this choice of torus, by Lemma \ref{lem:supp of spherical}, the torus elements that lies in the support of $\phi_K$ are in $\oA$. Also the torus elements appearing in the calculation are in $\overline{\SL}_r$. Recall that $Y_{Q,n}\cap Y^{sc}=Y_{Q,n}^{sc}=n_QY^{sc}$. As a consequence, the only patterns to consider are those where $d_i\equiv 0 \mod n_Q$. (See also \cite{Kaplan} Theorem 43.)

\begin{Lem}\label{lem:bounding the degree}
The monomial of the highest total degree in $\cw_0(0,\ochi)$ is at most the same as in $f(x_{12},\cdots,x_{k-1,k})$, up to a scalar.
\end{Lem}

\begin{Ex}[Continuation of Example \ref{ex:main section}]\label{ex:main section 2}
We continue with the set up in Example \ref{ex:main section}. A quick calculation shows that the monomial of the highest total degree should be $x_1^5x_2^4$.

Now let us check it from the Gelfand-Tsetlin description. We require that the $4$th and $7$th row to be as large as possible. So the maximal possible $4$th row is $(7 ~ 6 ~ 5 ~ 4 ~ 3 )$ and the maximal possible $7$th row is $(7~6)$. This gives $d_4\leq 3\times 5=15$ and $d_7\leq 6\times 2=12$. Therefore, the monomial with the highest total degree is again $x_1^5x_2^4$.
\[
\begin{pmatrix}
 7 &   &  6 &   &  5 &   &  4 &   &   3 &   &  2 &   &  1  &   & 0 \\
   & \ast  &  & \ast  &  & \ast  &  & \ast  &  & \ast  &  &  \ast &  & \ast &  \\
  & & \ast  &  & \ast  &  & \ast  &  & \ast  &  & \ast  &  &  \ast &  &   \\
   &   &  & 7  &  & 6  &  & 5 &  & 4  &  &  3 &  &  &  \\
   & & &  & \ast  &  & \ast  &  & \ast  &  & \ast  &  &  &  &   \\
    &   &  &   &  & \ast  &  & \ast  &  & \ast  &  &   &  &  &  \\
    &   &   &   &   &   &  7 &   &   6 &   &   &   &    &   &  \\
   &   &  &   &  &   &  & \ast  &  &  &  &   &  &  &  \\
\end{pmatrix}.
\]
\end{Ex}

\begin{proof}[Proof of Lemma \ref{lem:bounding the degree}]
We now seek the monomial of highest total degree. We consider patterns with the maximal entries $a_{e_{l+1}+1,j}$ for $1\leq l\leq k-1$ and $1\leq j\leq r-(e_{l+1}+1)$ possible. We now fix $l$. Note that $a_{e_{l+1}+1,1}\leq r-1$ as $a_{1,1}=r-1$. The maximal possible choice of row $e_{l+1}+1$ is
\[
(r-1,r-2,\cdots)
\]
Therefore, $d_{e_{l+1}+1}\leq e_{l+1}(\sum_{j=l+1}^{k}r_j)= (\sum_{j=1}^lr_j)(\sum_{j=l+1}^{k}r_j)$. As $d_{e_{l+1}+1}/n_Q$ must be an integer, this implies that
\[
d_{e_{l+1}+1}/n_Q\leq\left\lfloor\dfrac{(\sum_{j=1}^l r_j)(\sum_{j=l+1}^k r_j)}{n}\right\rfloor.
\]
The result now follows from \eqref{eq:strange condition}.
\end{proof}

By the above two results, we know that $\cw_0(0,\ochi)=cf(x_{12},\cdots,x_{k-1,k})$ for some constant $c$. It remains to compute a single coefficient of $f$. In \cite{Kaplan}, the highest monomial is used for this purpose. Here, we calculate the constant coefficient. We claim that only the lowest pattern contributes to the constant coefficient and the contribution is therefore $1$.

\begin{Ex}[Continuation of Example \ref{ex:main section 2}]
We are again in the setup of Example \ref{ex:main section}. We would like to show that only the lowest pattern contributes to the constant coefficient. First of all, we have $d_4=d_7=0$. This determines the $4$th and $7$th rows. These entries are as small as possible. Thus some entries in the other rows are determined. So far we have:
\[
\begin{pmatrix}
 7 &   &  6 &   &  5 &   &  4 &   &   3 &   &  2 &   &  1  &   & 0 \\
   & a_{21}  &  & a_{22}  &  & 4  &  & 3  &  & 2  &  &  1 &  & 0 &  \\
  & & a_{31}  &  & 4  &  & 3  &  & 2  &  & 1 &  &  0 &  &   \\
   &   &  & 4  &  & 3  &  & 2 &  & 1  &  &  0 &  &  &  \\
   & & &  & a_{51}  &  & a_{52}  &  & 1  &  & 0  &  &  &  &   \\
    &   &  &   &  & a_{61}  &  & 1  &  & 0  &  &   &  &  &  \\
    &   &   &   &   &   &  1 &   &   0 &   &   &   &    &   &  \\
   &   &  &   &  &   &  & a_{81}  &  &  &  &   &  &  &  \\
\end{pmatrix}.
\]
We next show that $d_i=0$ for all $i$. This determines the pattern completely. For instance,
\[
d_2=(a_{21}-6)+(a_{22}-5)\leq (a_{11}-6)+(a_{12}-5)=2.
\]
As $d_2\equiv 0\mod 3$, we must have $d_2=0$. The other cases can be proved similarly.
\end{Ex}

\begin{Lem}
Only the lowest pattern contributes to the constant coefficient of $\cw_0(0,\ochi)$.
\end{Lem}

\begin{proof}
To find the term contributing to the constant coefficient, we must have
\[
d_{e_{l+1}+1}=0,\qquad l=1,\cdots, d-1.
\]
Given a fixed $l$, this determines row $e_{l+1}+1$, which is
\begin{equation}\label{eq:entries in certain row}
(r-e_{l+1}-1,\cdots,1,0).
\end{equation}
The last $r-e_{l+1}$ entries from row  $e_{l}+2$ to row $e_{l+1}$ are also determined. They are \eqref{eq:entries in certain row} as well. We now determine the remaining coefficients. 

We now fix $1\leq l\leq  k$. We argue by induction to show $d_{e_l+i}=0$ for $i=1,\cdots, r_l$.  The case $i=1$ follows from our discussion above. We now assume that $d_{e_l+i}=0$. Then row $e_l+i$ is
\[
(r-e_l-i,\cdots, 1,0).
\]
In other words, $a_{e_l+i,j}=r-e_l-i-j+1$.
Thus,
\[
\begin{aligned}
d_{e_l+i+1}=&\sum_{j=1}^{r_l-i} a_{e_l+i+1,j}-(r-e_l-i-j)\\
\leq& \sum_{j=1}^{r_l-i} a_{e_l+i,j}-(r-e_l-i-j)\\
=& \sum_{j=1}^{r_l-i} (r-e_l-i-j+1)-(r-e_l-i-j)=r_l-i<n_Q.\\
\end{aligned}
\]
As $d_{e_l+i+1}\equiv 0 \mod n_Q$, we deduce that $d_{e_l+i+1}=0$.

We now conclude that $d_i=0$ for $i$ and the pattern must be the lowest pattern. This completes the proof.
\end{proof}

It is straightforward to see that the contribution of the lowest pattern is $1$. The proof of the base case is now complete.

\subsection{Proof of Theorem \ref{thm:final formula}: the general case}

We now prove the general case of Theorem \ref{thm:final formula}. It remains to show that for $\bw\in W(\bm)$,
\[
\cw_{\bw[0]}(0,\ochi)=g(\bw,0)\cw_{0}(0,\ochi).
\]
By Corollary \ref{cor:proportional}, it suffices to show that
\[
\dfrac{\tau(\bw_M,{}^{\bw_M^{-1}}\ochi,\bw[0],0)}{\tau(\bw_M,{}^{\bw_M^{-1}}\ochi,0,0)} =g(\bw,0).
\]
We now use Lemma \ref{lem:10} to calculate the left-hand side. Suppose $\ochi\sim (\ochi_1,\cdots, \ochi_k)$ for some anti-exceptional characters $\ochi_1,\cdots,\ochi_k$. We have
\[
\begin{aligned}
&\tau(\bw_M,{}^{\bw_M^{-1}}\ochi,\bw[0],0)\\
=&\prod_{i=1}^k \tau\left(\bw_{\GL_{r_i}},{}^{\bw_{\GL_{r_i}}^{-1}}\ochi,\bw_i[0],0\right)\\
=&\prod_{i=1}^k g(\bw_i,0)=g(\bw,0).
\end{aligned}
\]
This finishes the proof.




\end{document}